\documentclass[oneside,11pt]{amsart}
\usepackage[active]{srcltx}
\usepackage{amsmath,amssymb}
\newcommand{\vertiii}[1]{{\left\vert\kern-0.25ex\left\vert\kern-0.25ex\left\vert #1
    \right\vert\kern-0.25ex\right\vert\kern-0.25ex\right\vert}}
\usepackage{amsmath, amsfonts,amsthm,times,graphics}

 \makeatletter
\renewcommand*\subjclass[2][2000]{%
  \def\@subjclass{#2}%
  \@ifundefined{subjclassname@#1}{%
    \ClassWarning{\@classname}{Unknown edition (#1) of Mathematics
      Subject Classification; using '1991'.}%
  }{%
    \@xp\let\@xp\subjclassname\csname subjclassname@#1\endcsname
  }%
}
 \makeatother

\newtheorem{theorem}{Theorem}[section]

\newtheorem{lemma}[theorem]{Lemma}
\newtheorem{corollary}[theorem]{Corollary}

\theoremstyle{definition}
\newtheorem{definition}[theorem]{Definition}

\newtheorem{remark}[theorem]{Remark}
\numberwithin{equation}{section}
\newcommand{\abs}[1]{\lvert#1\rvert}

\newcounter{minutes}
\divide\time by 60
\newcounter{hours}
\multiply\time by 60 \addtocounter{minutes}{-\time}

\begin{document}

\title{On Riesz type inequalities for harmonic mappings on the unit disk}


\author{David Kalaj}
\address{Faculty of Natural Sciences and Mathematics, University of
Montenegro, Cetinjski put b.b. 81000 Podgorica, Montenegro}
\email{davidk@ac.me}

\def\thefootnote{}
\footnotetext{ \texttt{\tiny File:~\jobname.tex,
          printed: \number\year-\number\month-\number\day,
          \thehours.\ifnum\theminutes<10{0}\fi\theminutes }
} \makeatletter\def\thefootnote{\@arabic\c@footnote}\makeatother

\begin{abstract}
We prove some sharp inequalities for complex harmonic functions on the unit disk. The results extend a  M. Riesz conjugate function theorem and some well-known estimates for holomorphic functions. We apply some of results to the isoperimetric inequality for harmonic mappings.
\end{abstract}
\maketitle
\tableofcontents

\section{Introduction}
Let $\mathbf{U}$ denote the unit disk and $\mathbf{T}$ the unit circle in the complex plane.
For $p>1$, we define the Hardy class $\mathbf{h}^p$ as the class of harmonic mappings $f=g+\bar h$, where $g$ and $h$ are holomorphic mappings defined on the unit disk $\mathbf{U}\subset\mathbf{C},$ so that $$\|f\|_p=\|f\|_{\mathbf{h}^p}=\sup_{0<r< 1} M_p(f,r)<\infty,$$ where $$M_p(f,r)=\left(\int_{\mathbf{T}}|f(r\zeta)|^p d\sigma(\zeta)\right)^{1/p}.$$ Here $\sigma$ is probability measure on $\mathbf{T}$. The subclass of holomorphic mappings that belongs to the class $\mathbf{h}^p$ is denoted by $H^p$.

If $f\in \mathbf{h}^p$, then by \cite[Theorem~6.13]{axl}, there exists
$$f(e^{it})=\lim_{r\to 1} f(re^{it}),   a.e.$$  and $f\in L^p(\mathbf{T}).$
Then there hold
\begin{equation}\label{come}\|f\|^p_{{\mathbf{h}^p}}=\lim_{r\to 1}\int_{0}^{2\pi}|f(re^{it})|^p \frac{dt}{2\pi}= \int_{0}^{2\pi}|f(e^{it})|^p \frac{dt}{2\pi}.\end{equation}

Similarly we define the {\it Hardy space} $H^p$ of holomorphic
functions.

Let $1<p<\infty$ and let $\overline p=\max\{p,p/(p-1)\}$. Verbitsky in \cite{ver} proved the  following result.  If $f=u+iv\in H^p$ and $v(0)=0$, then \begin{equation}\label{ver}\sec(\pi/(2\overline p))\|v\|_p\leq\|f\|_p\leq\csc(\pi/(2\overline p))\|u\|_p,\end{equation} and both estimates are sharp. This result improves the sharp inequality \begin{equation}\label{pico}\|v\|_p\leq\cot(\pi/(2\overline p))\|u\|_p\end{equation} found by S. K. Pichorides (\cite{pik}). For the same problem for real line setting
we refer to the papers by L. Grafacos (\cite{graf}) and B. Hollenbeck,  N. J. Kalton, I. E. Verbitsky (\cite{studia}). We also refer to the paper by Essen  \cite{essen} for some related results.

We extend those results for the harmonic functions in Hardy class $\mathbf{h}^p$ on the unit disk $\mathbf{U}$. For a harmonic mapping $f=g+\overline{h}\in \mathbf{h}^p $ , $(hg)(0)=0$, we define the norm $\vertiii{\cdot}_p=\vertiii{\cdot}_{\mathbf{h}^p}$ as follows $$\vertiii{f}_{p}=\sup_{0<r<1}\left(\int_{\mathbf{T}}(|g(rz)|^2+|h(rz)|^2)^{p/2}d\sigma(z)\right)^{1/p}.$$
Thus, in view of \eqref{come}, we have that $$\vertiii{f}_{p}=\left(\int_{\mathbf{T}}(|g(z)|^2+|h(z)|^2)^{p/2}d\sigma(z)\right)^{1/p}.$$

Then we find the best constants $A_p$ and $B_p$ in the inequalities
\begin{equation}\label{pok}
\vertiii{f}_p\le A_p\|f\|_p
\end{equation}

\begin{equation}\label{kop}
\|f\|_p\le B_p \vertiii{f}_p.
\end{equation}
Namely we show in our main results  that $$A_p=\frac{1}{\left(1-\abs{\cos\frac{\pi}{p}}\right)^{1/2}}\ \ \ \ \text{(Theorem~\ref{kalaj})}$$  and      $$B_p=\sqrt{2}\cos{\frac{\pi}{2\overline{p}}}\ \ \ \ \text{(Theorem~\ref{kalaj1})}.$$
By taking $g=h$ in \eqref{pok} and \eqref{kop} (see Corollary~\ref{ver2} and Corollary~\ref{ver3} below) we deduce \eqref{ver}.
One of application of our result is the exact calculation of the norm of complex Hilbert transform on the unit disk (and on the unit circle) and on the upper half-plane (and on the real line).

Namely we show that the norm of the complex (periodic and non-peridic) Hilbert transforms $\mathcal{H}: L^p(\mathbf{T},\mathbf{C})\to L^p(\mathbf{T},\mathbf{C})$ and $\mathcal{H}: L^p(\mathbf{R},\mathbf{C})\to L^p(\mathbf{R},\mathbf{C})$,  is  \begin{equation}\label{pbar}\|\mathcal{H}\|_p =\cot \frac{\pi}{2\bar{p}}\ \ \ (\textrm{Theorem}~\ref{prente}, \ \ \textrm{Corollary}~\ref{hili}).\end{equation}
 As another application of our main results, we prove an isoperimetric type inequality for harmonic mappings $h$ defined on the Bergman space $\mathbf{b}^p$ on the unit disk, where $p$ is an even integer larger than $2$. Namely for $n\in\mathbf{N}$ and $n\ge 2$ we obtain \begin{equation}\label{est1}\|f\|_{\mathbf{b}^{2n}}\le \frac{1}{2} \csc\left[\frac{\pi }{4 n}\right]\|f\|_{\mathbf{h}^n}\ \ \ (\textrm{Theorem}~\ref{isop}).\end{equation}

\section{Main results}
The first main result is the following theorem
\begin{theorem}\label{kalaj}
Let $1<p<\infty$. Assume that $f=g+\bar h\in \mathbf{h}^p$ is a harmonic mapping on the unit disk with $\Re(g(0)h(0))= 0$
. Then we have the following sharp inequality
\begin{equation}\label{nesi}\left(\int_{\mathbf{T}}(|g|^2+|h|^2)^{p/2}\right)^{1/p}\le \frac{1}{\left(1-\abs{\cos\frac{\pi}{p}}\right)^{1/2}}\left(\int_{\mathbf{T}} |g+\bar h|^p\right)^{1/p}.\end{equation}
The proof given below works under the weaker condition: $\Re(g(0)h(0))\ge 0$ for $1<p\le 3$. The sharpness of the constant follows from the sharpness of the corollary below.
\end{theorem}

Here and in the sequel throughout the whole paper, we use the notation $$\int_{\mathbf{T}} f:=\int_{\mathbf{T}} f(z)d\sigma(z).$$

\begin{corollary}\cite{ver}\label{ver2}
Let $g=u+iv$ be a holomorphic function so that $v(0)=0$, then the sharp inequality \begin{equation}\label{aleks}\|g\|_{H^p}\le \frac{1}{\cos \frac{\pi}{2\overline{p}}}\| u\|_{\mathbf{h}^p}\end{equation} holds.
\end{corollary}
\begin{proof}[Proof of Corollary~\ref{ver}]
Let $1<p\le 2$. Since $v(0)=0$, it follows that $g(0)g(0)=u^2(0)\ge 0$, and so the inequality \eqref{aleks} follows by applying the previous theorem to real harmonic function $f=g+\bar g$ and by using the formula $${\sqrt{2}}{\left(1-\abs{\cos\frac{\pi}{p}}\right)^{-1/2}}=\frac{1}{\cos \frac{\pi}{2 \overline{p}}}.$$ If $p>2$, then we make use of inequality \eqref{eqeq} below. We have by using Jensen inequality (as in \cite{ver}), the following
\[\begin{split}\|g\|_{H^p}=\|\sqrt{u^2+v^2}\|_{L^p(\mathbf{T})}\le \left(\|u\|^2_{\mathbf{h}^p}+\|v\|^2_{\mathbf{h}^p}\right)^{1/2}
\le \left(\|u\|^2_{\mathbf{h}^p}+\sin^2\frac{\pi}{2\overline{p}}\|g\|^2_{H^p}\right)^{1/2}.
\end{split}\]
Therefore $$\|g\|^2_{H^p}(1-\sin^2\frac{\pi}{2\overline{p}})\le \|u\|^2_{\mathbf{h}^p}$$ and this implies the corollary.
\end{proof}
To motivate the following theorem notice the following simple sharp inequality $$|z+\bar w|\le \sqrt{2}(|z|^2+|w|^2)^{1/2}.$$ Thus we have
$$\|f\|_{\mathbf{h}^p}\le  \sqrt{2}\|\sqrt{|g|^2+|h|^2}\|_{L^p(\mathbf{T})}.$$ However the last  inequality is not sharp, and the sharp inequality has been given by the following theorem.\begin{theorem}\label{kalaj1}
Let $1<p<\infty$ and assume that $f=g + \bar h \in \mathbf{h}^p$ is a harmonic mapping on the unit disk with $\Re(g(0)\cdot h(0))\le 0$. Then we have the following sharp inequality
\begin{equation}\label{lopa}\|f\|_{\mathbf{h}^p}\le  \sqrt{2}\max\{\sin\frac{\pi}{2p},\cos\frac{\pi}{2p}\}\left(\int_{\mathbf{T}}(|g|^2+|h|^2)^{p/2}\right)^{1/p}.\end{equation}
\end{theorem}
\begin{remark}
If $p=2$, then inequalities \eqref{nesi} and \eqref{lopa} are opposite to each other because  $$\frac{1}{\left(1-\abs{\cos\frac{\pi}{p}}\right)^{1/2}}= 1=\sqrt{2}\max\{\sin\frac{\pi}{2p},\cos\frac{\pi}{2p}\}.$$ This is not a surprising fact, because the given integrals coincide if $\Re (h(0)g(0))=0$. In other words for every $f\in \mathbf{h}^2$, $\|h\|_{ \mathbf{h}^2}=\vertiii{h}_{ \mathbf{h}^2}$.
\end{remark}

\begin{corollary}\cite{ver}\label{ver3}
If $v$ is a real harmonic function with $v(0)=0$  and $g=u + i v$ is an analytic function, then for every $p>1$ we have the inequality
\begin{equation}\label{eqeq}
\|v\|_{\mathbf{h}^p}\le \sin \frac{\pi}{2\bar p}\|g\|_{H^p}.
\end{equation}
\end{corollary}
\begin{proof}[Proof of Corollary~\ref{ver3}]
By applying Theorem~\ref{kalaj1} to the real harmonic function $$v=f=\frac{g -\bar g}{2i}=-\frac{1}{2}(ig +\overline{ig}),$$ in view of the fact $$i^2g(0)g(0)=-g^2(0)\le 0,$$ and by using the simple formula $$\max\{\sin\frac{\pi}{2p},\cos\frac{\pi}{2p}\}=\sin\frac{\pi}{2\overline{p}}$$ we obtain \eqref{eqeq}.
\end{proof}
\subsection{Application to Hilbert transform}\label{sub}
\subsubsection{Hilbert transform on the unit disk and unit circle}
If $f=u+iv$ is a harmonic
function defined in the unit disk $\mathbf{U}$, then a harmonic function $\tilde f=\tilde u+i\tilde v$ is called the harmonic
conjugate of $f$ if $u+i\tilde u$ and $v+i\tilde v$ are analytic functions. Notice that $\tilde f$ is uniquely determined up to an additive constant. Let $f=g+\bar h=u+iv $ be a harmonic mapping, where $h$ and $g$ are holomorphic and $ h(0)=0$. Then $\tilde f := -(i g+\overline{i h})=\tilde u + i \tilde v$ is a harmonic conjugate of $f$ which we deal with in this paper. Namely $u=\Re(g+h)$, $v=\Im(g+h)$.
Further $\tilde u =\Re (-i (g-\bar h))=\Im(g-\bar h)$ and $\tilde v= \Im (-i (g-\bar h))=\Re (\bar h-g)$.
Thus $u + i \tilde u = \Re(g+h)+i \Im(g-\bar h)= g+ h$ and $v+i\tilde v= \Im(g+h)+i \Re (\bar h-g)=i(h-g)$.

Further if $f$ is real valued, i.e. if $f(z)=h(z)+\overline{h(z)}$, then $f(z)=h(z)+\overline{h(0)}+\overline{h(z)-h(0)}=g_1(z)+\overline{h_1(z)}$. Here $h_1(0)=0$. If $h(0)\neq 0$, then
$\tilde f(z)=i (g_1(z)-\overline{h_1(z)})$ is not real valued function, but $\hat{f}(z):=\tilde f(z)- \tilde f(0)$ is real valued with $\hat{f}(0)=0$. Then $\hat{f}$ is the harmonic conjugate with respect to standard meaning. Furthermore for every $z$
\begin{equation}\label{hph}|\hat f(z)|^2=|\tilde f(z)|^2-|\tilde f(0)|^2.\end{equation}

Let
$\chi$ be the boundary value of $f$ and assume that $\tilde \chi$ is the boundary value of $\tilde f$. Then $\tilde \chi$ is called the Hilbert transform of $\chi$ and we denote it by $\tilde\chi=H[\chi]$. Assume that $\tilde\chi\in L^1(\mathbf{T})$.

The (periodic) Hilbert transform of a function $\chi\in L^1(\mathbf T)$
is also given  by the formulas
\begin{equation}\label{hilbert2}\tilde\chi(\tau)=-
\frac 1\pi
\int_{0^+}^\pi\frac{\chi(\tau+t)-\chi(\tau-t)}{2\tan(t/2)}
dt.\end{equation}
and
\begin{equation}\label{hilbert1}\tilde\chi(\tau)=-i\sum_{k\in\mathbf{Z}}\mathrm{sign}(k) \widehat{\chi}(k) e^{ik\tau},\end{equation} where
$$\widehat{\chi}(k)=\frac{1}{2\pi}\int_{\mathbf{T}} \chi(e^{it}) e^{-ik t}dt$$ and $\mathrm{sign}(0)=1$.
 Here $\int_{0^+}^\pi \Phi(t) dt:=\lim_{\epsilon\to 0^+}
\int_{\epsilon}^\pi \Phi(t) dt.$ The integral in \eqref{hilbert2} is improper and
converges for a.e. $\tau\in[0,2\pi]$. If $P$ denote the Poisson extension on the unit disk, then we have $\widetilde{P[\chi](z)}=P[\tilde\chi](z)$. 
\subsubsection{Hilbert transform on the real line and half-plan}
Let $p>1$ and let $f\in L^p(\mathbf{R},\mathbf{C})$. Then the (nonperiodic) Hilbert transform of $f$ is defined by
$$\mathcal{H}[\phi](x)=\tilde \phi(x)=\frac{1}{\pi}\int_{-\infty}^\infty \frac{\phi(t)}{x-t} dt=-\frac{1}{\pi}\lim_{\epsilon\downarrow 0}\int_\epsilon^\infty\frac{\phi(x+t)-\phi(x-t)}{t} dt.$$ 

Further, the mapping $\phi$ induces a harmonic mapping defined on the upper half-plane $\mathbf{H}:=\{z\in\mathbf{C}: \Im z>0\}$, by the formula $f(z)= P[\phi](z)$, where $P$ is the Poisson integral on the upper half-plane.
 
Now let  $\mathbf{h}^p(\mathbf{H})$ be  the Hardy space on the upper half-plane, i.e. the class of  harmonic mappings $f$ defined on $\mathbf{H}$ so that $$\|f\|_{\mathbf{h}^p(\mathbf{H})}:=\sup_{y>0} \|f(x+i y)\|_{L^p(\mathbf{R})}<\infty.$$  If $f=P[\chi](z)$, then we have  $\|f\|_{\mathbf{h}^p(\mathbf{H})}=\|\chi\|_{L^p(\mathbf{R})}$ (\cite[Theorem~7.17]{axl}).  
Furthermore the harmonic mapping $\tilde w= P[\tilde \phi]$ is harmonic conjugate of $w$.

The following theorem, in view of \eqref{hph} extends the main result of S. K. Pichorides (\cite{pik})
\begin{theorem}\label{prente}
Assume that $p>1$ and $f$ is a complex harmonic mapping so that $f=g+\bar h \in \mathbf{h}^p(\mathbf{U})$ and $h(0)=0$. Then $\tilde f=i (g-\bar h)\in \mathbf{h}^p$, and we have the sharp inequality
\begin{equation}\label{hilbert}
\|\tilde f\|_{\mathbf{h}^p}\le \cot\frac{\pi}{2\bar p}\| f\|_{\mathbf{h}^p}.
\end{equation}
In other words the norm of the operator $$\mathcal{H}:\mathbf{h}^p(\mathbf{U},\mathbf{C})\to \mathbf{h}^p(\mathbf{U},\mathbf{C})\ \ \  ( \mathcal{H}:L^p(\mathbf{T},\mathbf{C})\to L^p(\mathbf{T},\mathbf{C}))$$ is equal to $$\|\mathcal{H}\|_p=\left\{
                                                                                    \begin{array}{ll}
                                                                                      \tan\frac{\pi}{2 p}, & \hbox{if $p\le 2$;} \\
                                                                                      \cot\frac{\pi}{2 p}, & \hbox{if $p>2$.}
                                                                                    \end{array}
                                                                                  \right.
$$
\end{theorem}

\begin{proof}
Let $f=g+\bar h=u+iv $ be a harmonic mappings that belongs to $\mathbf{h}^p$. Then $\tilde f = -i (g-\bar h)=\tilde u + i \tilde v$.
From Theorem~\ref{kalaj} and Theorem~\ref{kalaj1}, in view of the fact that $\vertiii{f}_{\mathbf{h}^p}=\vertiii{\tilde f}_{\mathbf{h}^p}$ and $\Re((-i g(0))
(-i h(0)))=-\Re(g(0)h(0))=0$,  we obtain that
$$\|\tilde f\|_{\mathbf{h}^p}\le A_p \vertiii{f}_{\mathbf{h}^p}\le A_p B_p \| f\|_{\mathbf{h}^p}.$$ The theorem follows from the equation
$$A_p B_p=\frac{\sqrt{2}\cos{\frac{\pi}{2\overline{p}}}}{{\left(1-\abs{\cos\frac{\pi}{p}}\right)^{1/2}}}= \cot \frac{\pi}{2\bar{p}}. $$
\end{proof}

\begin{remark}
The condition $h(0)=0$ of Theorem~\ref{prente} is not essential. Indeed, if $h(0)\neq 0$, then $$f(z)=g(z)+\overline{h(0)}+\overline{h(z)-h(0)}= g_1(z)+\overline{h_1(z)},$$ where $h_1(0)=0$. In this case $$\tilde f(z)=-i(g_1(z)-\overline{h_1(z)})=- i \left(g(z)+\overline{h(0)}-  \overline{h(z)-h(0)}\right).$$
\end{remark}

\begin{corollary}\label{hili} Let $p>1$ and let $\chi\in L^p(\mathbf{R},\mathbf{C})$ and $f=P[\phi]$. 
Then we have  the sharp inequalities  \begin{equation}
\|\tilde \chi\|_{L^p}\le \cot \frac{\pi}{2\bar{p}}\|\chi\|_{L^p},
\end{equation}
and
\begin{equation}\label{shaq}
\|\tilde f\|_{\mathbf{h}^p}\le \cot \frac{\pi}{2\bar{p}}\|f\|_{\mathbf{h}^p}.
\end{equation}
In other words 
$$\|\mathcal{H}: L^p(\mathbf{R},\mathbf{C})\rightarrow L^p(\mathbf{R},\mathbf{C})\|=\cot \frac{\pi}{2\bar{p}}$$ and
$$\|\mathcal{H}: \mathbf{h}^p(\mathbf{H},\mathbf{C})\rightarrow \mathbf{h}^p(\mathbf{H},\mathbf{C})\|=\cot \frac{\pi}{2\bar{p}}.$$
\end{corollary}
\begin{proof}[Proof of Corollary~\ref{hili}]
The relation \eqref{shaq}  follows from Theorem~\ref{prente} and the approach in the proof of a result of Zygmund in \cite[Chapter~XVI,~Theorem~3.8]{ZY}, where it is proved that the constant that appear in the case of periodic Hilbert transform could be taken in non-periodic Hilbert transform case as well. Moreover, as it is shown by Pichorides in \cite[Theorem~4.1]{pik}, the constant $\cot \frac{\pi}{2\bar p}$ is sharp for real valued functions, and so it is sharp for our complex functions as well.
\end{proof}

Let $\mathbf{b}^{p}$ denote the Bergman class of harmonic mappings defined on the unit disk, satisfying the condition
$$\|f\|_{\mathbf{b}^p}:=\left(\int_{\mathbf{U}}|f(z)|^p\frac{dxdy}{\pi}\right)^{1/p}<\infty.$$ For compressive study of this class we refer to the book \cite{hkz}.

By integrating the functions $U_r(z)= r|f(zr)|^p$ and $V_r(z)=r(|g(zr)|^2+|h(zr)|^2)^{p/2}$ over the unit circle $\mathbf{T}$, using the inequalities
\eqref{nesi} and \eqref{lopa}, and integrating for $r\in[0,1]$ we obtain the following result for the Bergman space $\mathbf{b}^p.$
\begin{corollary}\label{kalaj2}
Let $1<p<\infty$. Assume that $f=g+\bar h\in \mathbf{b}^p$ is a harmonic mapping on the unit disk with $\Re(g(0)h(0))= 0$
. Then we have the following inequalities
\begin{equation}\label{nesi1}\left(\int_{\mathbf{U}}(|g|^2+|h|^2)^{p/2}\right)^{1/p}\le \frac{1}{\left(1-\abs{\cos\frac{\pi}{p}}\right)^{1/2}}\left(\int_{\mathbf{U}} |g+\bar h|^p\right)^{1/p}\end{equation}
and
\begin{equation}\label{nesi2}\left(\int_{\mathbf{U}}(|g+\bar h|)^{p}\right)^{1/p}\le \sqrt{2}\sin \frac{\pi}{2\bar p}\left(\int_{\mathbf{U}} (|g|^2+ |h|^2)^{p/2}\right)^{1/p}.\end{equation}
The inequality \eqref{nesi2} does hold under weaker condition $\Re(g(0)h(0))\le 0$, and \eqref{nesi1} for $\Re(g(0)h(0))\ge 0$ and $p<3$.
\end{corollary}
\begin{remark}
We were not able to check if the  inequalities \eqref{nesi1} and \eqref{nesi2} are sharp or not for $p\neq 2$.  We want to emphasis the following fact. Some well-known  extremal functions that works for Hardy space, are not suitable for the Bergman space. The following example suggested by A. Calderon (see \cite{pik}) shows that \eqref{aleks} and \eqref{eqeq} are sharp. Namely if $g(z)=\left(\frac{1+z}{1-z}\right)^{2\gamma/\pi}$, $\abs{\arg\frac{1+z}{1-z}}\le \frac{\pi}{2}$, and $\gamma<\frac{\pi}{2p}$, then $g=u +iv\in h^p$. Further $|u|=\tan \gamma |v|$ almost everywhere on $\mathbf{T}$, but $|u|-\tan \gamma|v|>0$  everywhere on $\mathbf{U}$. This is why this example works for Hardy space but not for Bergman space.

\end{remark}
\subsection{Application to the isoperimetric inequality}
The starting point of this subsection  is the well known isoperimetric
inequality for Jordan domains and isoperimetric inequality for
minimal surfaces due to Carleman \cite{tc}. In that paper Carleman,
among the other results proved that if $u$ is harmonic and smooth in
$\overline {\mathbf{ U}}$ then $$\int_{\mathbf{ U}}e^{2u}dx dy\le
\frac{1}{4\pi}(\int_0^{2\pi}e^u dt)^2.$$ By using a similar approach
as Carleman, Strebel (\cite{ks}) proved the isoperimetric inequality
for holomorphic functions; that is if $f\in H^1(\mathbf{ U})$ then
\begin{equation}\label{isoper}\int_{\mathbf{ U}}|f(z)|^2 dxdy \le \frac{1}{4\pi} (\int_{\mathbf{T}}|f(e^{it})|dt)^2.\end{equation}
By using the normalized measures on $\mathbf{T}$ and $\mathbf{U}$, respectively, the previous inequality can be written in the form

\begin{equation}\label{isoper}\int_{\mathbf{ U}}|f(z)|^2  \le \ \left(\int_{\mathbf{T}}|f(z)|\right)^2.\end{equation}
This inequality has been proved
independently by Mateljevi\'c and Pavlovi\'c (\cite{mp}). In
\cite{hw}, F. Hang, X. Wang, X. Yan have made a certain generalizations for the space.

 Now we prove the following  theorem
\begin{theorem}\label{isop}
Let $f$ be a complex harmonic mapping defined on the unit disk and assume that $n\ge 2 $ is a positive integer. Assume that $f\in \mathbf{h}^{n}$, then $f\in \mathbf{b}^{2n}$ and we have the inequality \begin{equation}\label{est}\|f\|_{\mathbf{b}^{2n}}\le \frac{1}{2} \csc\left[\frac{\pi }{4 n}\right]\|f\|_{\mathbf{h}^n}.\end{equation}
\end{theorem}
\begin{remark}
 The proofs of the same statement for $n=2$ and $n=4$ can be found in \cite{ka1} and in \cite{kaba} respectively (where different approaches used, but applicable only for those two specific cases).

The proof here works only for positive integers $n\ge 2$, but probably the same estimate is true for every positive number $n>2$. On the other hand,  we where not able to check if the inequality \eqref{est} is sharp.
\end{remark}
A positive real function $u$ is called log-subharmonic, if $\log u$ is subharmonic.
First we formulate a lemma whose proof can be also deduced from \cite[Corollary~1.6.8]{lars}.
\begin{lemma}
The function $|a|^2+|b|^2$ is log-subharmonic, provided that $a$ and $b$ are analytic.
\end{lemma}
\begin{proof}
We need to show that $f(z)=\log(|a|^2+|b|^2)$ is subharmonic. By calculation we find $$
f_z= \frac{a'\bar a+b'\bar b}{|a|^2+|b|^2}$$ and so
\[\begin{split}f_{z\bar z}&= \frac{a'\bar a'+b'\bar b'}{|a|^2+|b|^2}-\frac{a'\bar a+b'\bar b}{|a|^2+|b|^2}
\frac{a\bar a'+b\bar b'}{|a|^2+|b|^2}\\&=\frac{(|a'|^2+|b'|^2)(|a|^2+|b|^2)-|\bar aa'+\bar bb'|^2}{(|a|^2+|b|^2)^2},\end{split}\] which is clearly positive.

\end{proof}

Now the isoperimetric inequality for log-subharmonic functions (e.g. \cite[Lemma~2.2]{ka}), states that, if $u$ is positive log-subharmonic function, then  $$\int_{\mathbf{U}}u^2\le \left(\int_{\mathbf{T}}u\right)^{2}.$$
Here as before, $$\int_{\mathbf U}f:=\frac{1}{\pi}\int_{\mathbf U}f(z){dxdy},\ \ \ z=x+iy,$$ and
$$\int_{\mathbf T}f:=\frac{1}{2\pi}\int_{\mathbf T}f(z){|dz|},\ \ \ z=x+iy.$$
Thus we infer that
\begin{lemma}\label{ipl}
For every positive number $p$ and analytic functions $a$ and $b$ defined on the unit disk $U$  we have that
$$\int_{\mathbf{U}}(|a|^2+|b|^2)^{2p}\le \left(\int_{\mathbf{T}}(|a|^2+|b|^2)^{p}\right)^2.$$
\end{lemma}

\begin{proof}[Proof of Theorem~\ref{isop}]
Without loos of generality  assume that $f(z) = g(z)+\overline{h(z)}$, where $h(0)=0$, and $g$ and $h$ are holomorphic on the unit disk.

Let
$$L=\int_{\mathbf{U}} (|g + \bar h|^2)^n=\int_{\mathbf{U}} (|g|^2 +  |h|^2+2\Re (gh))^n.$$
Then

\[\begin{split}
L&
 = \sum_{k=0}^n \binom{n}{k}\int_{\mathbf{U}} (|g|^2+|h|^2)^k (2\Re(gh))^{n-k}\\&\le \sum_{k=0}^n\binom{n}{k}\int_{\mathbf{U}} ((|g|^2+|h|^2)^n)^{k/n}
(\int_{\mathbf{U}} |2\Re(gh)|^n)^{(n-k)/n}.
\end{split}\]
Let $p\ge 2$ and let  $E_p = \cos\frac{\pi}{2p}.$

From Lemma~\ref{ipl} and Corollary~\ref{kalaj2} and Theorem~\ref{kalaj1} (Corollary~\ref{ver3}) we have
\[\begin{split}L  &  \le \sum_{k=0}^n \binom{n}{k} E_n^{n-k} (\int_{\mathbf{U}} (|g|^2+|h|^2)^n)^{k/n} (\int_{\mathbf{U}} (2|gh|)^n)^{(n-k)/n}
\\&\le \sum_{k=0}^n \binom{n}{k} E_n^{n-k} (\int_{\mathbf{T}} (|g|^2+|h|^2)^{n/2})^{2k/n} (\int_{\mathbf{T}} (2|gh|)^{n/2})^{2(n-k)/n}
\\&\le \sum_{k=0}^n \binom{n}{k} E_n^{n-k} (\int_{\mathbf{T}} (|g|^2+|h|^2)^{n/2})^{2k/n} (\int_{\mathbf{T}} (|g|^2+|h|^2)^{n/2})^{2(n-k)/n}
\\&= \sum_{k=0}^n \binom{n}{k} E_n^{n-k}  (\int_{\mathbf{T}} (|g|^2+|h|^2)^{n/2})^{2}
\\& =  \sum_{k=0}^n \binom{n}{k} E_n^{n-k} \frac{1}{\left(1-\abs{\cos\frac{\pi}{p}}\right)^{n}} (\int_{\mathbf{T}} |g+\bar h|^n)^{2}
\\& =  \frac{(1+E_n)^n}{\left(1-{\cos\frac{\pi}{n}}\right)^{n}} (\int_{\mathbf{T}} (|g+\bar h|^n))^{2}
\\& =  \frac{(1+\cos\frac{\pi}{2n})^n}{\left(1-{\cos\frac{\pi}{n}}\right)^{n}} (\int_{\mathbf{T}} |g+\bar h|^n)^{2}
.\end{split}\]
Thus $$\int_{\mathbf{U}} (|g + \bar h|^2)^n\le \frac{(1+\cos\frac{\pi}{2n})^n}{\left(1-{\cos\frac{\pi}{n}}\right)^{n}} (\int_{\mathbf{T}} |g+\bar h|^n)^{2}.$$ Further
$$\frac{1+\cos\frac{\pi}{2n}}{1-{\cos\frac{\pi}{n}}}=\frac{\cos ^2\frac{\pi}{4n }}{\sin^2\frac{\pi}{2n}}=\frac{1}{2\sin^2\frac{\pi}{4n}}.$$
This finishes the proof.
\end{proof}

\section{Strategy of the proofs}
As the authors of the paper did in \cite{verb1}, we use "pluri-subharmonic minorant".
\begin{definition}
A upper semi-continuous real function $u$ is called subharmonic in an open set $\Omega$ of complex plane, if for every compact subset $K$ of $\Omega$ and for every harmonic function $f$ defined on $K$, the inequality $u(z)\le f(z)$ for $z\in \partial K$ implies that $u(z)\le f(z)$ on $K$.
\end{definition}
A property which characterizes the subharmonic mappings is the sub-mean value property which states that. If $u$ is a subharmonic function defined on a domain $\Omega$, then for every closed disk $\overline{D(z_0,r)}\subset \Omega$, we have the inequality $$u(z_0)\le \frac{1}{2\pi r}\int_{|z-z_0|=r} {u(z)|dz|}.$$
\begin{definition} A function $u$ defined in an open set $\Omega\subset \mathbf{C}^n$ with values
in $[ - \infty, +\infty)$ is called plurisubharmonic if
\begin{enumerate}
\item $u$ is semicontinuous from above;
\item For arbitrary  $z,w \in\mathbf{C}^n$, the function $t\to  u(z + tw)$ is
subharmonic in the part $\mathbf{C}$ where it is defined.
\end{enumerate}

\end{definition}

\begin{definition} A pluri-subharmonic function $f$ is called a pluri-subharmonic minorant of $g$ on $\Omega$ if $f(z,w)\le g(z,w)$ for $(z,w)\in \Omega\subset\mathbf{C}^2$, and $f(z_0,w_0)=g(z_0,w_0)$, for a point $(z_0,w_0)\in \Omega.$\end{definition}

Let $p>1$.  The main task in the proof of main results is to find optimal positive constants $a_p$, $b_p$, $c_p$ and $d_p$ and pluri-subharmonic functions $\mathcal{F}_p(z,w)$ and $\mathcal{G}_p(z,w)$ for $z,w\in \mathbf{C}$, vanishing for $z=0$ or $w=0$, so that the inequalities
\begin{equation}\label{calf}
(|w|^2+|z|^2)^{p/2}\le a_p |w+\bar z|^p - b_p \mathcal{F}_p(z,w)
\end{equation}
\begin{equation}\label{calg}
|w+\bar z|^p\le c_p (|w|^2+|z|^2)^{p/2}  - d_p \mathcal{G}_p(z,w)
\end{equation}
are sharp.
\section{Proof of Theorem~\ref{kalaj}}
To begin assume that $1<p<2$. The other cases and constants that appear in this paper are found in a similar way, sometimes using Mathematica. The proofs of inequalities that we state are sometimes very technical, but also detailed.
We want to obtain a minimal positive constant $a_p$ and a positive constant $b_p$ so that
$$(|w|^2+|z|^2)^{p/2}\le a_p |w+\bar z|^p - b_p \mathcal{F}_p(z,w),$$ for  all complex numbers $z$ and $w$, where
$$\mathcal{F}_p(z,w)=\Re((wz)^{p/2}).$$
Let us chose  $w=1$, $z=1$.

Then we find the smallest positive constant $a_p$ in the inequality

$$-2^{p/2}+a_p (2+2 \cos t)^{p/2}-b_p \cos\left[\frac{p t}{2}\right]\ge 0.$$
Let
$$\omega(t):=-2^{p/2}+a (2+2 \cos t)^{p/2}-b \cos\left[\frac{p t}{2}\right].$$

We should find an appropriate positive constant $b=b_p$ and the minimal constant $a=a_p$ so that
$\omega(t)\ge 0 $ for $t\in[0,2\pi]$.
First of all $$\omega'(t_\circ)=-a p (2+2 \cos[t_\circ])^{-1+\frac{p}{2}} \sin[t_\circ]+\frac{1}{2} b p \sin\left[\frac{p t_\circ}{2}\right].$$
So if $\omega'(t_\circ)=0$, then $$b=\frac{a (2+2 \cos[t_\circ])^{p/2} \csc\left[\frac{p t_\circ}{2}\right] \sin[t_\circ]}{1+\cos[t_\circ]}.$$ (The condition $\omega'(t_\circ)=0$ means that $t_\circ$ should be the local and global minimum of $\omega$).

Now for such $t_\circ$, we have that $$\omega(t_\circ)=-2^{p/2}+a (2+2 \cos[t_\circ])^{p/2}-\frac{a (2+2 \cos[t_\circ])^{p/2} \cot\left[\frac{p t_\circ}{2}\right] \sin[t_\circ]}{1+\cos[t_\circ]}.$$
We chose $a$ and $t_\circ$ by the following two conditions
$$-2^{p/2}+a (2+2 \cos[t_\circ])^{p/2}=0$$ and $$\frac{a (2+2 \cos[t_\circ])^{p/2} \cot\left[\frac{p t_\circ}{2}\right] \sin[t_\circ]}{1+\cos[t_\circ]}=0.$$
(Those two conditions mean that $\omega(t_\circ)=0$ is the minimum of $\omega$).

The only solutions are $t_\circ=\pm \frac{\pi}{p}$ and $$a=\left(1+\cos\frac{\pi }{p}\right)^{-p/2}.$$ Further we find $$b=2^{p/2} \tan\frac{\pi }{2 p}.$$
Then we prove
\begin{lemma}\label{nice} Let $1<p\le 2$. For every two complex numbers $z$ and $w$ we have
\begin{equation}\label{fsh}(|w|^2+|z|^2)^{p/2}\le a_p |w+\bar z|^p - b_p \mathcal{F}_p(z,w),\end{equation} where
$$a_p=\left(1+\cos\frac{\pi }{p}\right)^{-p/2}, \ \ \ \ \ b_p=2^{p/2} \tan\frac{\pi }{2 p}$$ and
$$\mathcal{F}_p(z,w)=\Re((wz)^{p/2});$$ where for $\zeta=\rho e^{i\theta}$, \begin{equation}\label{rez}\Re(\zeta^{p/2}):=\rho^{p/2}\left\{
\begin{array}{ll}
\cos\left[\frac{\theta p}{2}\right], & \hbox{if $|\theta|\le \pi$;} \\
\cos\left[\frac{\theta p}{2}-p\pi\right], & \hbox{if $\pi \le \theta\le 2\pi$;} \\
\cos\left[\frac{\theta p}{2}+p\pi\right], & \hbox{if $-2\pi \le \theta\le -\pi$.}
\end{array}
\right.\end{equation}
Moreover $\Re(\zeta^{p/2})$ is subharmonic on $\mathbf{C}$ and $\mathcal{F}_p(z,w)$ is pluri-subharmonic on $\mathbf{C}^2.$ Furthermore the equality in \eqref{fsh} is attained if and if $|w|=|z|$ and $\arg(wz)=\frac{\pi}{p}\mod \pi$.
\end{lemma}
The last statement follows from a similar statement stated in \cite[Remark~2.3]{verb1}.
The inequality statement of Lemma~\ref{nice} follows from the following lemma, by dividing \eqref{fsh} with $\max\{|z|^p,|w|^p\}$ and taking $r=\min\{\frac{|z|}{|w|},\frac{|w|}{|z|}\}$.
\begin{lemma}\label{posa} For $r>0$ and $t\in[-\pi,\pi]$ and $p\in(1,2)$
$$G(r,t)=-\left(1+r^2\right)^{p/2}+ \left(\frac{1+r^2+2 r \cos t}{1+\cos\frac{\pi }{p}}\right)^{p/2}-2^{p/2} r^{p/2} \cos\frac{p t}{2} \tan\frac{\pi }{2 p}\ge 0.$$
\end{lemma}
We also need the following lemma.
\begin{lemma}\label{hard}
For $p>4$ define
$$\phi(\theta)=\left\{
              \begin{array}{rr}
                -\cos \frac{p}{2}(\frac{\pi}{2}-|\theta|), & \hbox{ if $\frac{\pi}{2}-2\frac{\pi}{p}\le|\theta|\le \frac{\pi}{2}$ ;} \\
                \max\{|\cos \frac{p}{2}(\frac{\pi}{2}-\theta)|,|\cos \frac{p}{2}(\frac{\pi}{2}+\theta)|\}, & \hbox{if $|\theta|\le \frac{\pi}{2}-2\frac{\pi}{p}$  ;} \\

              \end{array}
            \right.$$  and $\vartheta(\theta)=\vartheta_p(\theta)$ as follows
$$\vartheta(\theta):=\left\{
                     \begin{array}{ll}
                       \phi(\theta), & \hbox{if $|\theta|\le \pi/2$;} \\
                       \phi(\pi-|\theta|), & \hbox{if $\pi/2\le |\theta|\le \pi$.}
                     \end{array}
                   \right.$$
For $2<p\le 4$ define \begin{equation}\label{conj}\vartheta(\theta):=- \cos\frac{p}{2}(\pi-|\theta|).\end{equation}
Extend $\vartheta$ on $[-2\pi,2\pi]$ so that:  for $\pi \le |\theta|\le 2\pi$, $\vartheta(\theta):=\vartheta(\abs{\theta}-\pi)$.

Let $p>2$ and $$a_p = (1-\cos\frac{\pi}{p})^{-p/2}$$ and $$b_p=2^{p/2} \cot\frac{\pi }{2 p}. $$
Then for complex numbers $z=re^{it}$ and $w=Re^{is}$ and for $p\ge 4$ we have \begin{equation}\label{ale}(|z|^2+|w|^2)^{p/2}\le a_p |z+\bar w|^p - b_p (r R)^{p/2}\vartheta(s+t-\pi/2).
\end{equation}
For $2\le p\le 4$
\begin{equation}\label{ale1}(|z|^2+|w|^2)^{p/2}\le a_p |z+\bar w|^p - b_p (r R)^{p/2}\vartheta(s+t).\end{equation}
The equality is attained in \eqref{ale} if $|z|=|w|$ and $\arg(wz)=\pi/2+\frac{\pi}{p}$. On the other hand equality is attained in \eqref{ale1} for  $|z|=|w|$ and $\arg(wz)=\frac{\pi}{p}$.
\end{lemma}
We postpone the proofs of Lemma~\ref{nice} and Lemma~\ref{hard} and prove the fact that $\vartheta_p$ induces a subharmonic (and a pluri-subharmonic) function for every $p>1$.
\begin{lemma}\label{lemsub}
For $p\ge 2$, $z=|z|e^{i\theta}$ the function $$\Phi_p(z)= -|z|^{p/2}\cos\frac{p}{2}(\pi-|\theta|),
$$ is subharmonic in $\mathbf{C}$.
\end{lemma}
\begin{proof}[Proof of Lemma~\ref{lemsub}]
Let $2\le p\le 4$. $z_0=re^{i\theta}\in \mathbf{C}\setminus\{0\}$. If $\theta=0$, then near $z_0$, $\Phi_p(z)=\max\{-|z|^{p/2}\cos\frac{p}{2}(\pi-\theta), -|z|^{p/2}\cos\frac{p}{2}(\pi+\theta)\}$. Since $F_+(z)=-|z|^{p/2}\cos\frac{p}{2}(\pi-\theta)$ and $F_{-}=-|z|^{p/2}\cos\frac{p}{2}(\pi+\theta)$ are localy harmonic, it follows that $\Phi_p$ is subharmonic near $z_0$. If $\theta\neq 0$, then $\Phi_p$ coincides with $F_+$ or $F_{-}$ near $z_0$. Finally, since $$\frac{1}{2\pi r}\int_{|z|=r}\Phi_p(z) |dz|=-\frac{4 r^{p/2} \sin  \left[\frac{p \pi }{2}\right]}{ p}\ge 0=\Phi_p(0),$$ we obtain that $\Phi_p$ is subharmonic in $\mathbf{C}$. The case $p>4$ follows from
\end{proof}

\begin{lemma}\label{lemsub2}\cite[Lemma~3]{verb1}
For $p\ge 4$, $z=|z|e^{i\theta}$ the function $$\Phi_p(z)=|z|^{p/2}\vartheta_p(\theta-\pi/2)$$ is subharmonic in $\mathbf{C}$.
\end{lemma}
\begin{proof}[Proof of Lemma~\ref{lemsub2}]
The function $\Phi_p(z)$ coincides with $\Phi_{p/2}(z/i)$ from the paper \cite{verb1}. For the completeness include its proof.
Let $z_0=re^{i\theta}\in \mathbf{C}\setminus\{0\}$. If $\theta\neq 0$,  then near $z_0$,  $\Phi_p$ coincides with a harmonic function, and so is subharmonic in $z_0$. If $\theta=0$, then $\Phi_p$ is equal to the maximum of several harmonic functions of the form $u(re^{i\theta}) = r^{p/2}\cos(p/2(\theta_0 +\theta))$.
Finally, since
$\vartheta_p(\pi/p-x)=-\vartheta(\pi/p+x)$ for $x\in[\pi/2-2\pi/p,\pi/2]$ and $\vartheta_p(\pi/2-x)=\vartheta(\pi/2+x)$, we obtain $$ \int_{\pi/2-2\pi/p}^{\pi/2+2\pi/p}\vartheta_p(x)dx=0.$$ Thus
$$\frac{1}{2\pi r}\int_{|z|=r}F(z) |dz|=\frac{2r^{p/2}}{2\pi r}\int_{0}^{\pi/2-2\pi/p}\vartheta_p(x) dx> 0=F(0),$$ we obtain that $F$ is subharmonic in $\mathbf{C}$.
\end{proof}

\begin{proof}[Proof of Theorem~\ref{kalaj}]
In view of \eqref{rez}, Lemma~\ref{lemsub} and Lemma~\ref{lemsub2}, for  $z=|z|e^{it},w=|w|e^{is}\in \mathbf{C}$ define the function   $$\mathcal{F}_p(z,w)=\left\{
                                               \begin{array}{ll}
                                                 \Re((zw)^{p/2}), & \hbox{if $1<p\le 2$;} \\
                                                 \Phi_p(zw), & \hbox{if $p> 2$.}
                                               \end{array}
                                             \right.
$$ Then $\mathcal{F}_p$ is plurisubharmonic for every $p>1$. Let $1<p\le 2$ and assume that $f=g+\bar h$, where $g$ and $h$ are holomorphic function on the unit disk. Then from Lemma~\ref{nice} and Lemma~\ref{hard}, we have $$(|g(z)|^2+|h(z)|^2)^{p/2}\le a_p |g(z)+\overline{h(z)}|^p - b_p \mathcal{F}_p(g(z),h(z)),$$ where
$$a_p=\left(1-\abs{\cos\frac{\pi }{p}}\right)^{-p/2}$$ and  $$b_p=2^{p/2} \tan\frac{\pi }{2 p}.$$
Then $$\int_{\mathbf{T}} (|g(z)|^2+|h(z)|^2)^{p/2}\le a_p \int_{\mathbf{T}} |f(z)|^p - b_p \int_{\mathbf{T}}\mathcal{F}_p(g(z),h(z)).$$
Let $\theta=\arg(g(0)h(0))$. As $\mathcal{F}_p(g(z),h(z))$ is subharmonic, by sub-mean inequality we have that $$\int_{\mathbf{T}}\mathcal{F}_p(g(z),h(z))\ge \mathcal{F}_p(g(0),h(0))=|g(0)h(0)|^p\cos\left[p\frac{|\theta|}{2}\right]\ge 0,$$ if $\theta\in[-\pi/2,\pi/2]$.
If $2<p\le 4$, then we also have that $$\int_{\mathbf{T}} (|g(z)|^2+|h(z)|^2)^{p/2}\le a_p \int_{\mathbf{T}} |f(z)|^p - b_p \int_{\mathbf{T}}\mathcal{F}_p(g(z),h(z)),$$ in conjunction with the inequality $$ \int_{\mathbf{T}}\mathcal{F}_p(g(z),h(z))\ge \mathcal{F}_p(g(0),h(0))=-|g(0)h(0)|^p\cos\left[p\frac{\pi-|\theta|}{2}\right]=:R.$$
Further
$$R\ge \left\{
                                                         \begin{array}{ll}
                                                            0, & \hbox{if $p\le 3$ and $\Re(g(0)h(0))\le 0$, i.e. $\theta\in[-\pi/2,\pi/2]$;} \\
                                                            0, & \hbox{if $3\le p\le 4$ and $\Re(g(0)h(0))=0$, i.e. if $\theta=\pm \pi/2$.}
                                                         \end{array}
                                                       \right.
$$

If $p\ge 4$, then  $$\int_{\mathbf{T}} (|g(z)|^2+|h(z)|^2)^{p/2}\le a_p \int_{\mathbf{T}} |g(z)+\overline{ih(z)}|^p - b_p \int_{\mathbf{T}}\mathcal{F}_p(g(z),h(z)). $$
Further we have $$\int_{\mathbf{T}}\mathcal{F}_p(g(z),h(z))\ge \mathcal{F}(g(0),h(0))=|g(0)h(0)|^{p/2}|\cos \frac{p\pi}{4}|\ge 0,$$ if $\Re(g(0)h(0))=0$, i.e. if $\theta=\pm \pi/2$. This finishes the proof.
\end{proof}

\begin{proof}[Proof of Lemma~\ref{posa}]
Let $S=[-\pi,\pi]\times [0,1]$. We prove that $$\min_{(t,r)\in S} G(t,r)=G\left(\frac{\pi}{p},1\right)=0.$$

We need the following claim.
Let $t\in[-\pi,\pi]$ and let $p\in(1,2)$. Let  $$P(r)= -\left(\frac{1+r^2}{2r}\right)^{p/2}+\left(\frac{1+r^2+2 r \cos t}{2r(1+ \cos\frac{\pi }{p})}\right)^{p/2}-\cos\frac{p t}{2} \tan\frac{\pi }{2 p}.$$ Then $p$ is decreasing for $r\in[0,1]$.
We have that $$P'(r)=Q\left( \left(\frac{1+r^2}{1+r^2+2 r \cos t }\right)^{1-p/2}-\left(1+\cos\frac{\pi }{p}\right)^{\frac{p}{2}}\right),$$ where
$$Q=\left(1+r^2\right)^{-1+\frac{p}{2}}2^{-1-\frac{p}{2}} p r^{-1-\frac{p}{2}} \left(-1+r^2\right)\left(1+\cos\frac{\pi }{p}\right)^{-p/2}.$$
So we need to show that $$\left(\frac{1+r^2}{1+r^2+2 r \cos t }\right)^{1-p/2}\ge \left(1+\cos\frac{\pi }{p}\right)^{\frac{p}{2}}.$$
Since
$$\frac{d}{dr}\frac{1+r^2}{1+r^2+2 r \cos t}=\frac{2 \left(-1+r^2\right) \cos t}{\left(1+r^2+2 r \cos t\right)^2}$$ we see that for
$\cos t\ge 0$, we have that \[\begin{split}\left(\frac{1+r^2}{1+r^2+2 r \cos t }\right)^{1-p/2}&\ge \left(\frac{2}{1+1+2  \cos t }\right)^{1-p/2}\\&\ge 2^{p/2-1}\ge\left(1+\cos\frac{\pi }{p}\right)^{\frac{p}{2}}. \end{split}\]
In order to prove the last inequality consider the function $$\psi(p)=\frac{2^{-1+\frac{p}{2}}}{\left(1+\cos\frac{\pi }{p}\right)^{p/2}}$$ and show that $\psi(p)\ge 1$. Then
$$\psi'(p)=\frac{ \sec ^p\frac{\pi }{2 p}\left(2 p \log \left[\sec \frac{\pi }{2 p}\right]-\pi  \tan \frac{\pi }{2 p}\right)}{4 p}.$$
Further $$\left(2p \log \left[\sec \frac{\pi }{2 p}\right]-\pi  \tan \frac{\pi }{2 p}\right)\le \left(4 \log \left[\sec \frac{\pi }{2 p}\right]-\pi  \tan \frac{\pi }{2 p}\right).$$ Let $y=\cos \frac{\pi }{2 p}$. Then
$$4\log \left[\sec \frac{\pi }{2 p}\right]-\pi  \tan \frac{\pi }{2 p}=\phi(y):=4\log \frac{1}{y}-\pi  \frac{\sqrt{1-y^2}}{y}.$$
Next $$\phi'(y)=\frac{-4 y+\frac{\pi }{\sqrt{1-y^2}}}{y^2}.$$
Now $-4 y+\frac{\pi }{\sqrt{1-y^2}}\ge 0$ if and only if $$\omega(y):=(1 - y^2) 16 y^2 - \pi^2\le 0.$$
As $$\omega'(y)=32 y-64 y^3=32y(y-1/\sqrt{2})(y+1/\sqrt{2}),$$ we obtain that $\omega(y)\le \omega(1/\sqrt{2})=4-\pi^2<0$. Hence $\phi'(y)\ge 0$. So $$\psi'(p)\le\sec ^p\left[\frac{\pi }{2 p}\right] \frac{ \phi(y)}{4 p}\le \sec ^p\left[\frac{\pi }{2 p}\right]\frac{ \phi(1)}{4 p}=0 .$$
This implies that $\psi(p)\ge \psi(2)=1$.

If $\cos t\le 0$, then $$\left(\frac{1+r^2}{1+r^2+2 r \cos t }\right)^{1-p/2}\ge \left(\frac{2}{1+1+2  \cos t }\right)^{1-p/2}\ge 1\ge\left(1+\cos\frac{\pi }{p}\right)^{\frac{p}{2}} .$$
Thus we proved that $P$ is decreasing.

Since $$G(r,t)=\frac{P(r)}{(2r)^{p/2}},$$ we have that $$G(r,t)\ge \frac{P(1)}{(2r)^{p/2}}.$$ So it remains to prove that $P(1)\ge 0$. This means that
we need to prove the inequalities
\begin{equation}\label{first}-1+\left(\frac{1+\cos t}{1+\cos\frac{\pi }{p}}\right)^{p/2}-\cos\left[\frac{p t}{2}\right] \tan \frac{\pi }{2 p}\ge 0, |t|\le \pi\end{equation}

\begin{equation}\label{second}-1+\left(\frac{1+\cos t}{1+\cos\frac{\pi }{p}}\right)^{p/2}-\cos\left[\frac{p t}{2}-p\pi\right] \tan \frac{\pi }{2 p}\ge 0, \pi\le t\le  2\pi\end{equation}

\begin{equation}\label{third}-1+\left(\frac{1+\cos t}{1+\cos\frac{\pi }{p}}\right)^{p/2}-\cos\left[\frac{p t}{2}+p\pi\right] \tan \frac{\pi }{2 p}\ge 0,-2\pi\le  t\le -\pi.\end{equation}

Prove \eqref{first}. Then \eqref{second} and \eqref{third} follows from \eqref{first}, by changing the variables $t=2\pi +t'$ or $t=-2\pi+t'$.
By taking the substitution $s=t/2$, \eqref{first}  reduces to the inequality
$$-1+\left(\frac{\cos s}{\cos\left[\frac{\pi }{2p}\right]}\right)^{p}-\cos\left[ps\right] \tan \frac{\pi }{2 p}\ge 0$$
for $s\in(0,\pi/2).$ But this is the same as \cite[Lemma~1]{ver}, which proof is including here (see below Lemma~\ref{versi}), because it is missing in \cite{ver}, and seems to the author that is not trivial.

\end{proof}

\begin{proof}[Proof of Lemma~\ref{hard}]
Consider two cases

\textbf{The case $p\ge 4$.}
Prove that \begin{equation}\label{ale23}(|z|^2+|w|^2)^{p/2}\le a_p |z-i\bar w|^p - b_p (r R)^{p/2}\vartheta(s+t),
\end{equation}
Without loss of generality assume that $R=1$, $s=0$  and $r<1$.
Prove that \[\begin{split}H&=-\left(1+r^2\right)^{p/2}+(2r)^{p/2}  \cot\frac{\pi }{2 p}\left(\cos \left[\frac{p}{2}  \left(\frac{\pi }{2}-x\right)\right]\right)\\&+
 \left(\frac{1+r^2-2 r \sin x}{1-\cos \frac{\pi }{p}}\right)^{p/2}\ge 0\end{split}\] on the interval $\frac{\pi }{2}-\frac{2 \pi }{p}\leq x\leq \frac{\pi }{2}$.

Let $$G(r)=-\left(\frac{1+r^2}{2r}\right)^{p/2}+\left(1-\cos \frac{\pi }{p}\right)^{-p/2} \left(\frac{1+r^2-2 r \sin x}{2r}\right)^{p/2}.$$ Then $H\ge 0$ if and only if $$G(r)\ge \cot\frac{\pi }{2 p}\left(\cos \left[\frac{p}{2}  \left(\frac{\pi }{2}-x\right)\right]\right).$$
Let $$a=\frac{1+r^2}{2r}.$$ Then $$H(a)=G(r)=-a^{p/2}+\left(1-\cos \frac{\pi }{p}\right)^{-p/2} \left(a- \sin [x]\right)^{p/2}.$$
We have $$H'(a)=-\frac{1}{2}pa^{-1+\frac{p}{2}}+\frac{1}{2} p \left(1-\cos \frac{\pi }{p}\right)^{-p/2} (a-\sin x)^{-1+\frac{p}{2}}.$$
Then $H'(a)< 0$ if \begin{equation}\label{xl}x< \sin^{-1}\left[1-\left(1-\cos \frac{\pi }{p}\right)^{\frac{p}{p-2}}\right]\end{equation} and thus $G(r)\ge G(1)$. Show that $$G(1)\ge -\cot\frac{\pi }{2 p}\left(\cos \left[\frac{p}{2}  \left(\frac{\pi }{2}-x\right)\right]\right).$$
 We have to show that $$-1+2^{-p/2} \left(1-\cos \frac{\pi }{p}\right)^{-p/2} (2-2 \sin [x])^{p/2}\ge -\cot\frac{\pi }{2 p}\cos \left[\frac{p}{2}  \left(\frac{\pi }{2}-x\right)\right].$$
Let $y=\pi/2-x$, then the previous inequality can be written as

$$-1+ \left(\frac{1- \cos y}{1-\cos \frac{\pi }{p}}\right)^{p/2}\ge -\cot\frac{\pi }{2 p}\cos \left[\frac{p}{2} y\right].$$

Here $0\le y\le 2\pi/p$.

Let

$$\chi(y):=\left(\frac{1-\cos y}{1-\cos \frac{\pi }{p}}\right)^{-p/2} \left(1-\cos \left[\frac{p y}{2}\right] \cot \frac{\pi }{2 p}\right)$$ and prove that $\chi(y)\le 1$.

We have that
\[\begin{split}
\chi'(y)&=\frac{p  \left(\left(-1+\cos \left[\frac{p y}{2}\right] \cot \frac{\pi }{2 p}\right) \sin y+(1-\cos y) \cot \frac{\pi }{2 p} \sin \left[\frac{p y}{2}\right]\right)}{2 \left(\frac{1-\cos y}{1-\cos \frac{\pi }{p}}\right)^{p/2} (1-\cos y)}
\\& = -\frac{p \left(\frac{1-\cos y}{1-\cos \frac{\pi }{p}}\right)^{-p/2} (\sin y-2 \cos \left[\frac{1}{2} (p-1) y\right] \sin \left[\frac{y}{2}\right] \cot \frac{\pi }{2 p})}{2 (1-\cos y)}
\\&
 = -\frac{p \left(\frac{1-\cos y}{1-\cos \frac{\pi }{p}}\right)^{-p/2} (\cos \frac{y}{2}- \cos \left[\frac{p-1}{2}  y\right] \cot \frac{\pi }{2 p})}{2 \sin \frac{y}{2}}.\end{split}\]
So we should prove that the function $$\zeta(t)=(\cos t- \cos \left[{(p-1)}  t\right] \cot \frac{\pi }{2 p})$$ is negative  for $t\in[0,\pi/(2p)]$ and positive for $t\in[\pi/(2p),\pi/p]$.

We have $$\zeta'(t)=\sin t+(1-p) \cot \frac{\pi }{2 p} \sin [(p-1) t]\le \sin(t)-\sin((p-1)t),$$ because $$(1-p) \cot \frac{\pi }{2 p}<-1.$$
Now for $0<(p-1)t<\pi/2$, it is clear that $-\sin(t)+\sin((p-1)t)> 0$. If $(p-1)t>\pi/2$ then $0<\pi-(p-1)t<\pi/2$ and so
$-\sin(t)+\sin(\pi-(p-1)t)> 0$ because $\pi-(p-1)t\ge t$ i.e. $t\le \pi/p$. Thus $\zeta$ has only one stationary point which is equal to $\pi/(2p)$. This implies that
$\chi'(y)\ge 0$ if $y\le \pi/p$ and $\chi'(y)\le 0$ if $\frac{\pi}{p}\le y\le 2\frac{\pi}{p}$. So $\max\chi(y)=\chi(\pi/p)=1$ and the proof is finished for the case $0\le y\le 2\pi/p$. So the case $$\frac{\pi }{2}-\frac{2 \pi }{p}\leq x\leq  \sin^{-1}\left[1-\left(1-\cos \frac{\pi }{p}\right)^{\frac{p}{p-2}}\right]$$ for $p>4$  has been completed.

If \begin{equation}\label{pix} \sin^{-1}\left[1-\left(1-\cos \frac{\pi }{p}\right)^{\frac{p}{p-2}}\right]<x\le\frac{\pi}{2},\end{equation} then $H'(a_p)=0$ for $$ a_p=\frac{\sin x}{1-\left(1-\cos \frac{\pi }{p}\right)^{\frac{p}{p-2}}}, $$ and $$H(a_p)=\min_{a\ge 1} H(a).$$
We should show that $$H(a_p)\ge -\cot\frac{\pi }{2 p}\cos \left[\frac{p}{2}  \left(\frac{\pi }{2}-x\right)\right],$$
i.e. if
$$x> \sin^{-1}\left[1-\left(1-\cos \frac{\pi }{p}\right)^{\frac{p}{p-2}}\right]$$ then
\[\begin{split}-\left({\left(1-\cos \frac{\pi }{p}\right)^{\frac{p}{2-p}}}-1\right)^{1-\frac{p}{2}} & \left(1-\cos \frac{\pi }{p}\right)^{-p/2}\sin^{p/2} x\\&\ge -\cot\frac{\pi }{2 p}\cos \left[\frac{p}{2}  \left(\frac{\pi }{2}-x\right)\right]\end{split}\]
or
\[\begin{split}\left(-1+\left(1-\cos \frac{\pi }{p}\right)^{\frac{p}{2-p}}\right)^{1-\frac{p}{2}} &\left(1-\cos \frac{\pi }{p}\right)^{-p/2}\sin^{p/2} x\\&\le \cot\frac{\pi }{2 p}\cos \left[\frac{p}{2}  \left(\frac{\pi }{2}-x\right)\right].\end{split}\]
Now $$W(y)=:\frac{\cos \left[\frac{p y}{2}\right] \cot \frac{\pi }{2 p}}{\cos^{p/2} y}$$ is monotone decreasing for
\begin{equation}\label{inw} 0\le y\le \pi/2-\sin^{-1}\left[1-\left(1-\cos \frac{\pi }{p}\right)^{\frac{p}{p-2}}\right].\end{equation}
Indeed  $$W'(y)=\frac{1}{2} p \cos^{-1-\frac{p}{2}} y \cot \frac{\pi }{2 p} \sin\left[y-\frac{p y}{2}\right].$$
Further in view of \eqref{inw}  we have that \begin{equation}\label{web}-\frac{\pi}{2} \le \frac{1}{2} (2-p) \cos^{-1}\left[1-\left(1-\cos\left[\frac{\pi }{p}\right]\right)^{\frac{p}{p-2}}\right] \le y-\frac{py}{2}\le 0.\end{equation}
Namely $$-\frac{\pi}{2} \le \frac{1}{2} (2-p) \cos^{-1}\left[1-\left(1-\cos\left[\frac{\pi }{p}\right]\right)^{\frac{p}{p-2}}\right]$$ if and only if
$$\cos\left[\frac{\pi }{p-2}\right]<1-\left(1-\cos\left[\frac{\pi }{p}\right]\right)^{\frac{p}{p-2}},\ \ p\ge 4.$$ The last inequality is trivial because $$\cos\left[\frac{\pi }{p-2}\right]\le \cos\left[\frac{\pi }{p}\right], \ \ \ \ p\ge 4.$$

Thus
$$W(y)\ge W\left(\pi/2-\sin^{-1}\left[1-\left(1-\cos \frac{\pi }{p}\right)^{\frac{p}{p-2}}\right]\right). $$
Thus we have to show that
\[\begin{split}&\left(-1+\left(1-\cos \frac{\pi }{p}\right)^{\frac{p}{2-p}}\right)^{1-\frac{p}{2}} \left(1-\cos \frac{\pi }{p}\right)^{-p/2}\\&\le
\frac{\cos \left[\frac{1}{2} p \cos^{-1} \left[1-\left(1-\cos \frac{\pi }{p}\right)^{\frac{p}{p-2}}\right]\right] \cot \frac{\pi }{2 p}}{\left(1-\left(1-\cos \frac{\pi }{p}\right)^{\frac{p}{p-2}}\right)^{p/2}}\end{split}\]

or what is the same
\begin{equation}\label{same}1-\left(1-\cos \frac{\pi }{p}\right)^{\frac{p}{p-2}}\le \cos \left[\frac{1}{2} p \cos^{-1} \left[1-\left(1-\cos \frac{\pi }{p}\right)^{\frac{p}{p-2}}\right]\right] \cot \frac{\pi }{2 p}.\end{equation}

Prove instead that
\begin{equation}\label{prima}1-\left(1-\cos \frac{\pi }{p}\right)^{\frac{p}{p-2}}\le \cos \frac{\pi }{2 p}\end{equation} and

\begin{equation}\label{seconda1}\cos \left[\frac{1}{2} p \cos^{-1} \left[1-\left(1-\cos \frac{\pi }{p}\right)^{\frac{p}{p-2}}\right]\right] \ge \sin\frac{\pi }{2 p}.\end{equation}
Let $x=1/p$. Then the inequality \eqref{prima} is equivalent with

$$-\log \left[1-\cos \left[\frac{\pi  x}{2}\right]\right]+\frac{\log [1-\cos [\pi  x]]}{1-2 x}\ge 0, \ \ 0\le x\le 1/4,$$ or what is the same

$$\delta(x):=-({1-2 x})\log \left[1-\cos \left[\frac{\pi  x}{2}\right]\right]+{\log [1-\cos [\pi  x]]}\ge 0, \ \ 0\le x\le 1/4.$$
As $$\delta'(x)=\pi  x \cot \left[\frac{\pi  x}{4}\right]+2 \log \left[2 \sin\left[\frac{\pi  x}{4}\right]^2\right]-\frac{1}{2} \pi  \tan \left[\frac{\pi  x}{4}\right]<0,$$ we have $$\delta(x)\ge \delta(1/4)\ge 0.$$

Further
$$\cos \left[\frac{1}{2} p \cos^{-1} \left[1-\left(1-\cos \frac{\pi }{p}\right)^{\frac{p}{p-2}}\right]\right] \ge \sin\frac{\pi }{2 p}$$

if and only if $$\frac{\pi}{2}-\frac{1}{2} p \cos^{-1} \left[1-\left(1-\cos \frac{\pi }{p}\right)^{\frac{p}{p-2}}\right]\ge \frac{\pi}{2p}$$

if and only if

$${\pi}- p \cos^{-1} \left[1-\left(1-\cos \frac{\pi }{p}\right)^{\frac{p}{p-2}}\right]\ge \frac{\pi}{p}$$

if and only if \begin{equation}\label{hardy}\cos\left(\frac{\pi-\frac{\pi}{p}}{p}\right)\le 1-\left(1-\cos \frac{\pi }{p}\right)^{\frac{p}{p-2}},\ \ \ p\ge 4.\end{equation}
We must emphasis that the proof of inequality \eqref{hardy} below is rather long.
After the substitution $y=1/p$ the last inequality reduces to the inequality
$$\beta(y):=(1-2y)\log [1-\cos [y (\pi -\pi  y)]]-\log [(1-\cos [\pi  y])\ge 0,\ \ 0< y\le 1/4.$$
Then $$\beta'(y)=-\pi  \cot \left[\frac{\pi  y}{2}\right]-\pi  (1-2 y)^2 \cot \left[\frac{1}{2} \pi  (-1+y) y\right]-2 \log [1-\cos [\pi  (-1+y) y]]$$
and
\[\begin{split}-\beta''(y)&= \pi  (6-12 y) \cot \left[\frac{1}{2} \pi  (1-y ) y\right]\\&+\frac{1}{2} \pi^2  \left(-\csc\left[\frac{1}{2} \pi \text{  }y\right]^2+(1-2 y )^3 \csc\left[\frac{1}{2} \pi  (1-y ) y\right]^2\right).\end{split}\]

In order to continue let us prove the following lemma
\begin{lemma}\label{conti}
For $0< x\le \pi/4$ we have
\begin{equation}\label{prima}
1+\frac{1}{x^2}-\csc^2 x\ge 0,
\end{equation}
and 
\begin{equation}\label{seconda}
\cot x - \left(\frac{1}{x} - \frac{x}{2}\right)\ge 0.
\end{equation}

\end{lemma}
\begin{proof}[Proof of Lemma~\ref{conti}]
Prove first \eqref{prima}. It is equivalent with the inequality $$v(x):=\sin x \sqrt{1+\frac{1}{x^2}}\ge 1.$$ 

Further $$v(x)= \frac{\tan x}{\sqrt{1+\tan^2 x}}\frac{\sqrt{1+x^2}}{x}$$ which is clearly grater or equal to $1$ because $\tan x\ge x$, and $y/\sqrt{1+y^2}$ increases.

Prove now \eqref{seconda}.
First of all $$\cot x = \frac{1}{x}-\frac{1}{3} x-\frac{1}{45} x^3 -\dots =\sum_{n=0}^\infty \frac{(-1)^{n} 2^{2n} B_{2n}}{(2n)!} x^{2n-1},$$ were $B_{2n}$ are Bernoulli numbers. Further if $$A(x)=\sum_{n=2}^\infty \frac{(-1)^{n+1} 2^{2n} B_{2n}}{(2n)!} x^{2n-1},$$ then $$\cot x - \left(\frac{1}{x} - \frac{x}{2}\right)= \frac{x}{6}- A(x)=x\left(\frac{1}{6} -\frac{A(x)}{x}\right).$$

Since $\frac{(-1)^{n+1} 2^{2n} B_{2n}}{(2n)!}>0$ for $n>1$ we have $$\frac{A(x)}{x}=\frac{-\cot x + \frac{1}{x}-\frac{1}{3} x}{x}\le \frac{A(\pi/4)}{\pi/4}=\frac{4 \left(-1+\frac{4}{\pi }-\frac{\pi }{12}\right)}{\pi }<\frac{1}{6}.$$
This implies \eqref{seconda}.
\end{proof}
Further from \eqref{prima} and the simple inequality  $\csc x\ge \frac{1}{x},$ for $0\le x\le \pi/2$, we obtain 

\[\begin{split}\bigg(-\csc\left[\frac{1}{2} \pi \text{  }y\right]^2&+(1-2 y )^3 \csc\left[\frac{1}{2} \pi  (1-y ) y\right]^2\bigg)
\\& \ge - \csc\left[\frac{\pi  y}{2}\right]^2+ (1-2 y)^3 \left(\frac{1}{2} \pi  (1-y) y\right)^{-2}
\\&\ge -\left(\frac{4}{\pi^2 y^2}+1\right)+ (1-2 y)^3 \left(\frac{1}{2} \pi  (1-y) y\right)^{-2}
\\&
=-\frac{4}{\pi ^2 y^2}+\frac{4 (1-2 y)^3}{\pi ^2 (1-y)^2 y^2}-1.\end{split}\]
By using now \eqref{seconda} and the previous estimate we have that $$-\beta''(y)\ge \frac{8-28 y-\pi ^2 y+16 y^2-\pi ^2 y^2+14 \pi ^2 y^3-27 \pi ^2 y^4+21 \pi ^2 y^5-6 \pi ^2 y^6}{\pi  (-1+y)^2 y}$$ $$\ge \frac{8-28 y-\pi ^2 y+16 y^2-\pi ^2 y^2+14 \pi ^2 y^3-27 \pi ^2 y^4}{\pi  (-1+y)^2 y}.$$
Next if $$\gamma(y):=8-28 y-\pi ^2 y+16 y^2-\pi ^2 y^2+14 \pi ^2 y^3-27 \pi ^2 y^4$$ then $$\gamma'(y)=-28-\pi ^2+32 y-2 \pi ^2 y+42 \pi ^2 y^2-108 \pi ^2 y^3$$ and $$\gamma''(y)=32-2 \pi ^2+84 \pi ^2 y-324 \pi ^2 y^2>0,\ \  y\in[0,1/4].$$ So $\gamma'(y)$ is strictly increasing. Since $$\gamma'(1/4)=-20-\frac{9 \pi ^2}{16}<0$$ it follows that $\gamma$ is increasing. Thus $\gamma(y)\ge \gamma(1/4)=2-\frac{51 \pi ^2}{256}>0.$

Thus $\beta''(y)\le 0$ and so $\beta'$ is strictly decreasing. Since $\beta'(1/4)<0<\beta'(0)=+\infty$, it follows that $\beta$ increases on an interval $[0,r_0]$ and decreases on $[r_0,1/4]$. Since $\beta(0)>0$ and $ \beta(1/4)>0$, we obtain that $\beta$ is positive. This finishes the proof of \eqref{hardy}.

Assume now that  $0\le |x|\le \frac{\pi }{2}-\frac{2 \pi }{p}.$

Since $$\frac{\pi }{2}-\frac{2 \pi }{p}<\sin^{-1}\left[1-\left(1-\cos \frac{\pi }{p}\right)^{\frac{p}{p-2}}\right]$$ because

 $$\left(1-\cos \frac{\pi }{p}\right)^{\frac{p}{p-2}}+\cos \left[\frac{2 \pi }{p}\right]<1$$ we have again that $G(r)\ge G(1).$
Proceeding as in the case $\pi/2-2\pi/p\le x\le \pi/2$, and using the substitution $y=\pi/2-x$, this case reduces to showing that
$$\phi_1(y):=\left(\frac{1-\cos y}{1-\cos \frac{\pi }{p}}\right)^{-p/2} \left(1+\abs{\cos \left[\frac{p y}{2}\right]} \cot \frac{\pi }{2 p}\right)\le 1$$ for $2\pi/p\le y\le \pi-2\pi/p$ and

$$\phi_2(y):=\left(\frac{1-\cos y }{1-\cos \frac{\pi }{p}}\right)^{-p/2} \left(1+ \abs{\cos \left[\frac{p (y+\pi)}{2}\right]}\cot \frac{\pi }{2 p}\right)\le 1$$
 for $2\pi/p\le y\le \pi-2\pi/p.$

If $2\pi/p\le y\le \pi-2\pi/p$, then $\pi\le \frac{p}{2} y \le \frac{\pi p}{4}$ and thus there is a number $y'$ such that $\frac{p}{2} y'\in(\pi/2,\pi)$ and
$$\abs{\cos\frac{py}{2}}=-\cos \frac{p}{2} y'.$$
Then $\cos y\le \cos y'$, because  $0\le y'\le \frac{2\pi}{p}\le y\le {\pi}$ and thus $\phi_1(y)\le \chi(y')$ which is according to the previous case less or equal to $1$. Similarly we establish that $\phi_2(y)\le 1$.

\textbf{The case $2\le p\le 4$.}

Let $0\le x\le \pi$. $$G(r)=-\left(\frac{1+r^2}{2r}\right)^{p/2}+\left(1-\cos \frac{\pi }{p}\right)^{-p/2} \left(\frac{1+r^2+2 r \cos x}{2r}\right)^{p/2}.$$
Then \eqref{ale1} if and only if $$G(r)\ge -\cot\frac{\pi }{2 p}\cos \left[\frac{p}{2}  \left({\pi }-x\right)\right].$$

Let $$a=(1+r^2)/(2r).$$
Then $$H(a)=G(r)=-a^{p/2}+\left(1-\cos \frac{\pi }{p}\right)^{-p/2} \left(a+ \cos x\right)^{p/2}.$$

We have $$H'(a)=-\frac{1}{2}pa^{-1+\frac{p}{2}}+\frac{1}{2} p \left(1-\cos \frac{\pi }{p}\right)^{-p/2} (a+\cos x)^{-1+\frac{p}{2}}.$$
Then $H'(a)< 0$ if $$0<x< \cos^{-1}\left[1-\left(1-\cos \frac{\pi }{p}\right)^{\frac{p}{p-2}}\right] $$ and thus $G(r)\ge G(1)$.
So we need to prove that
$$\left(1-\cos \frac{\pi }{p}\right)^{-p/2} \left(1+ \cos x\right)^{p/2}\ge 1-\cot\frac{\pi }{2 p}\cos \left[\frac{p}{2}  \left({\pi }-x\right)\right]$$
i.e.
$$\left(\frac{1+ \cos x}{1-\cos \frac{\pi }{p}}\right)^{p/2}\ge 1-\cot\frac{\pi }{2 p}\cos \left[\frac{p}{2}  \left({\pi }-x\right)\right].$$
Let $$\chi_1(x)=\left(\frac{1+\cos x}{1-\cos \frac{\pi }{p}}\right)^{-p/2} \left(1-\cos \left[\frac{1}{2} p (\pi -x)\right] \cot \frac{\pi }{2 p}\right).$$ We should prove that $\chi_1(x)\le 1$.
We have that  $$\chi_1'(x)=\frac{p \left(\frac{1+\cos x}{1-\cos \frac{\pi }{p}}\right)^{-p/2} \left(-2 \cos \left[\frac{1}{2} (p-1) (\pi -x)\right] \cos \left[\frac{x}{2}\right] \cot \frac{\pi }{2 p}+\sin x\right)}{2 (1+\cos x)}.$$
By taking the substitution $y=\pi-x$, we arrive at the equality $\chi_1(y)=\chi(y)$ from the case $p>4$. The rest of the proof is the same as in the case $p>4$.

If $$\pi>x\ge \cos^{-1}\left[-1+\left(1-\cos \frac{\pi }{p}\right)^{\frac{p}{p-2}}\right] $$ then $$a_p=\frac{\cos x}{-1+\left(1-\cos \frac{\pi }{p}\right)^{\frac{p}{p-2}}}$$ is the only stationary point of $H(a)$. So we need to show that $$H(a_p)\ge  -\cot\frac{\pi }{2 p}\left(\cos \left[\frac{p}{2}  \left({\pi }-x\right)\right]\right),$$ which in view of substitution
$x=\pi-y$ is equivalent with the inequality

$$\cos \left[\frac{p y}{2}\right] \cot \frac{\pi }{2 p}+\left(\frac{\cos y \csc\frac{\pi }{2 p}^2}{2 \left(1-\cos \frac{\pi }{p}\right)^{-\frac{p}{p-2}}-2}\right)^{p/2}\geq \left(\frac{\cos y}{1-\left(1-\cos \frac{\pi }{p}\right)^{\frac{p}{p-2}}}\right)^{p/2}$$
or with its equivalent form
$$\frac{\cos \left[\frac{p y}{2}\right] \cot \frac{\pi }{2 p}}{\cos^{p/2} y}\geq\left(-1+\left(1-\cos \left[\frac{\pi }{p}\right]\right)^{\frac{p}{2-p}}\right)^{1-\frac{p}{2}}\left(1-\cos \left[\frac{\pi }{p}\right]\right)^{-p/2}$$ for
\[\begin{split}0\le y &\le \pi-\cos^{-1}\left[-1+\left(1-\cos \frac{\pi }{p}\right)^{\frac{p}{p-2}}\right]\\&=\pi/2- \sin^{-1}\left[1-\left(1-\cos \frac{\pi }{p}\right)^{\frac{p}{p-2}}\right]. \end{split}\]
Thus we have to prove \eqref{same} for $2<p\le 4$, i.e. $$\frac{2 \cos^{-1} \left[\left(1-\left(1-\cos \frac{\pi }{p}\right)^{\frac{p}{p-2}}\right) \tan \frac{\pi }{2 p}\right]}{p}\geq \cos^{-1} \left[1-\left(1-\cos \frac{\pi }{p}\right)^{\frac{p}{p-2}}\right].$$

Let $s=1-\left(1-\cos \frac{\pi }{p}\right)^{\frac{p}{p-2}}$ and $t= \frac{2}{p}$. Then

$$\frac{2 \cos^{-1} \left[\left(1-\left(1-\cos \frac{\pi }{p}\right)^{\frac{p}{p-2}}\right) \tan \frac{\pi }{2 p}\right]}{p}\ge t \cos^{-1}(ts),$$ and we should prove that $$t \cos^{-1}(ts)\ge \cos^{-1} s.$$ Let $H(t,s)={t \cos^{-1}(ts)}-{\cos^{-1} s}, \ \ \ (t,s)\in[0,1]\times[0,1].$ First of all
$$\lim_{p\to 2} 1-\left(1-\cos \frac{\pi }{p}\right)^{\frac{p}{p-2}}= 1-e^{-\pi/2}.$$
Further  $$H_s'(t,s)=\frac{1}{\sqrt{1-s^2}}-\frac{t^2}{\sqrt{1-s^2 t^2}}\ge 0.$$
If $2<p<3$, then for $t\in[2/3,1]$ and $s\in[1-e^{-\pi/2},1]$,
\[\begin{split}H(t,s)&\ge\eta(t):= H(t,1-e^{-\pi/2})\\&=-\cos^{-1}\left[1-e^{-\pi /2}\right]+t \cos^{-1}\left[\left(1-e^{-\pi /2}\right) t\right]\ge  0.\end{split}\] Namely for $c=1-e^{-\pi/2}$,
$$\eta''(t)= \frac{c \left(-2+c^2 t^2\right)}{\left(1-c^2 t^2\right)^{3/2}}<0.$$ Thus,
$$\eta(t)\ge \frac{1}{2}(\eta(2/3)+\eta(1)).$$ Hence
$$H(t,c)\ge \min\{H(2/3,c), H(1,c)\}=0.$$
If $3\le p\le 4$, then $t\in[1/2,2/3]$ and $s\in[\frac{7}{8},1].$ Then $$H(t,s)\ge \lambda(t)=H(t,7/8)=-\cos^{-1} \left[\frac{7}{8}\right]+t \cos^{-1} \left[\frac{7 t}{8}\right]>0.$$ Namely $$\lambda''(t)=\frac{-896+343 t^2}{\left(64-49 t^2\right)^{3/2}}< 0.$$ Thus $$\lambda(t)\ge \min\left\{\frac{2}{3} \cos^{-1}\left[\frac{7}{12}\right]-\cos^{-1}\left[\frac{7}{8}\right],\frac{1}{2} \cos^{-1}\left[\frac{7}{16}\right]-\cos^{-1}\left[\frac{7}{8}\right]\right\}>0.$$
\end{proof}

\section{Proof of Theorem~\ref{kalaj1}}
Define  $\vartheta_1(\theta)=\vartheta_{1,p}(\theta)$ as follows

For $\theta\in[0,2\pi]$ define

$$\phi_1(\theta)=\left\{
              \begin{array}{ll}
                -\cos {p\theta/2}, & \hbox{ if $0\le \theta\le 2\pi/p$ ;} \\

                -\cos {p/2}(2\pi-\theta), & \hbox{ if $2\pi-2\pi/p\le\theta\le 2\pi$ ;} \\
                \max\{|\cos {\frac{p\theta}{2}}|,|\cos {\frac{p}{2}}(2\pi-\theta)|\}, & \hbox{if $\frac{2 \pi }{p}\le \theta\le  \left(2-\frac{2}{p}\right) \pi $  ;} \\

              \end{array}
            \right.$$

$$\vartheta_1(\theta):=\left\{
                     \begin{array}{ll}
                       \phi(\theta), & \hbox{if $0\le \theta\le 2\pi $;} \\
                       \phi(-\theta), & \hbox{if $-2\pi\le \theta\le 0$.}
                     \end{array}
                   \right.$$
We will prove the theorem by using the following  lemmas
\begin{lemma}\label{l2} Let $p>2$. Then for complex numbers $z=|z|e^{it}$ and $w=|w|e^{is}$ we have
$$|z+\bar w|^p\le c_p (|z|^2+|w|^2)^{p/2}-d_p r^{p/2}\vartheta_1(t+s),$$ where $\vartheta_1$ is defined above, and $$c_p=\left[\sqrt{2}\cos \frac{\pi }{2 p}\right]^p$$ and $$d_p=\cos^{p-1}\frac{\pi}{2p}\sin\frac{\pi}{2p}.$$ This inequality is sharp. The equality is attained for $|z|=|w|\neq 0$ and $t+s\equiv\frac{\pi}{p}\mod \pi$.
\end{lemma}
\begin{lemma}\label{ari}
Let $p>2$. Then  the sharp inequality  \begin{equation}\label{sharp1}\left(1+r^2+2 r \cos t\right)^{\frac{p}{2}}\le 2^{\frac{p}{2}} \left(1+r^2\right)^{\frac{p}{2}} \cos^p\frac{\pi }{2 p}-\cos^{p-1}\frac{\pi}{2p}\sin\frac{\pi}{2p}r^{\frac{p}{2}}\vartheta_{1}(t)\end{equation} hold. The equality is attained if and only if $r=1$ and $t=\pm\frac{\pi}{p}$.

\end{lemma}
Then we prove that
\begin{lemma} Let $1<p<2$. Then for complex numbers $z$ and $w$ we have
$$|z+\bar w|^p\le c_p (|z|^2+|w|^2)^{p/2}-d_p |zw|^{p/2}\cos\frac{p}{2}(\pi-\abs{t+s}).$$
\end{lemma}
which is equivalent with the following lemma:
\begin{lemma}\label{vanesa}
Let $1<p<2$ and let $$c_p=(\sqrt{2}\sin\frac{\pi}{2p})^p$$ and $$d_p=\left(2-2 \cos \left[\frac{\pi }{p}\right]\right)^{p/2} \cot \left[\frac{\pi }{2 q}\right].$$
 Then  $$\left(1+r^2+2 r \cos t\right)^{p/2}\le c_p \left(1+r^2\right)^{p/2} -d_p r^{p/2}\cos(\frac{p}{2}(\pi-\abs{t}).$$ 
\end{lemma}

We postpone the proofs of Lemma~\ref{ari} and Lemma~\ref{vanesa}, and prove the theorem.
\begin{proof}[Proof of Theorem~\ref{kalaj1}]
For $p>1$ and $z=re^{i\theta}$ define
$$\Psi_p(z)=\left\{
  \begin{array}{ll}
    r^{p/2}\cos\left[\frac{p}{2}(\pi-\abs{\theta})\right], & \hbox{if $p<2$;} \\
    r^{p/2}\vartheta_{1}(t), & \hbox{if $p>2$.}
  \end{array}
\right.$$ Prove that for $p>1$, $\Psi_p$ is subharmonic on $\mathbf{C}$.
Let $1<p<2$.
Notice that $$\max\{r^{p/2} \cos[p/2(\pi -\theta)],r^{p/2} \cos[p/2(\pi+\theta)]\}=r^{p/2}\cos(\frac{p}{2}(\pi-\abs{\theta}).$$ Thus $$\Psi_p(z)=r^{p/2}\cos(\frac{p}{2}(\pi-\abs{\theta})$$ is subharmonic in $z\neq 0$. The subharmonicity at $z=0$ is verified by proving sub-mean inequality:
$$\frac{1}{2r\pi}\int_{-\pi}^{\pi}\Psi_p(re^{it})dt =\frac{4 \sin\left[\frac{p \pi }{2}\right]}{2r\pi p}\ge \Psi_p(0)=0.$$
For $p>2$, the proof of the fact that the function $\Psi_p$ is subharmonic on $\mathbf{C}$ is similar to the proof of subharmonicity of $\Phi_p$ in  Lemma~\ref{lemsub2}, so we skip the details.
Let $\mathcal{G}_p(z,w)=\Psi_p(zw)$. Then $\mathcal{G}_p$ is plurisubharmonic on $\mathbf{C}^2$. Thus $$\mathcal{K}(z)=\mathcal{G}_p(g(z),h(z))$$ is subharmonic on the unit disk. From Lemma~\ref{l2} and Lemma~\ref{ari}, we have

\begin{equation}\label{tan}|g(z)+\overline{h(z)}|^p\le c_p (|g(z)|^2+|h(z)|^2)^{p/2} -d_p \mathcal{K}(z).\end{equation} By integrating \eqref{tan} over $r\mathbf{T}$, $0<r<1$ and letting $r\to 1^-$, we obtain

$$\int_{\mathbf{T}} |g(z)+\overline{h(z)}|^p\le c_p \int_{\mathbf{T}} (|g(z)|^2+|h(z)|^2)^{p/2}- d_p\int_{\mathbf{T}}  \mathcal{K}(z).$$

Since $\Re(g(0)h(0))\le 0$, it follows that $\theta=\arg (g(0)h(0))\in (\pi/2,3\pi/2)$. Further for $p\ge 4$, $2\pi/p\le \pi/2\le 2\pi-2\pi/p$, and thus  $\vartheta_1(\theta)\ge 0$ and so $$\mathcal{K}(0)=|g(0)h(0)|^{p/2} \vartheta_1(\theta)\ge 0.$$ Thus $$\int_{\mathbf{T}}  \mathcal{K}(z)\ge 0.$$ This implies  that $$\int_{\mathbf{T}} |g(z)+\overline{h(z)}|^p\le c_p \int_{\mathbf{T}} (|g(z)|^2+|h(z)|^2)^{p/2}.$$

If $2\le p\le 4$, then $\pi/2\le 2\pi/p$  and thus $\theta\le 2\pi/p$ or $2\pi-2\pi/p\le \theta\le 2\pi$. Then $-\cos {p\theta/2}\ge 0$, and as before $$\int_{\mathbf{T}}  \mathcal{K}(z)\ge 0.$$

If $1\le p\le 2$, $$\mathcal{K}(0)=|g(0)h(0)|^{p/2} \cos p/2(\pi-\pi)=|g(0)h(0)|^{p/2}\cos\frac{p}{2}(\pi-\theta)\ge 0.$$ This finishes the proof.
\end{proof}

\begin{proof}[Proof o Lemma~\ref{ari}]
Define  $$P(r)=r^{-p/2} \left(-2^{p/2} \left(1+r^2\right)^{p/2} \cos^p\frac{\pi }{2 p}+\left(1+r^2+2 r \cos t\right)^{p/2}\right).$$

We should prove that $$P(r)\le -\cos^{p-1}\frac{\pi}{2p}\sin\frac{\pi}{2p}\vartheta_1(t).$$
We first have
$$P'(r)=\frac{p(r^2-1)}{r^{p/2}} \left(\left(1+r^2+2 r \cos t\right)^{\frac{p}{2}-1}-2^{p/2} \left(1+r^2\right)^{\frac{p}{2}-1} \sin^p\frac{\pi }{2 p}\right).$$
Then $P'(r)\ge 0$ if
\begin{equation}\label{mer} -\sec t\left(1- \left(\sqrt{2}\sin  \left[\frac{\pi }{2 p}\right]\right)^{\frac{2p}{p-2}}\right)\ge 1\end{equation} and thus  $$P(r)\le P(1)=\left(-2^p \cos^p\frac{\pi }{2 p}+(2+2 \cos t)^{p/2}\right)=2^p(\cos^p\frac{t}{2}-\cos^p\frac{\pi }{2 p}).$$
To continue, notice that in \cite[Lemma~2]{ver}, has been defined the function $\varphi=\varphi_p$, which coincides with our function $\vartheta_{2p}$. From \cite[Lemma~2]{ver} we obtain
\[\begin{split}2^p\cos^p\frac{t}{2}&=2^p|\sin\frac{t+\pi}{2}|^p\\&\le 2^{p}  \cos^p\frac{\pi }{2 p}-\cos^{p-1}\frac{\pi}{2p}\sin\frac{\pi}{2p}\varphi_{p}(\frac{t+\pi}{2})
\\&= 2^{p}  \cos^p\frac{\pi }{2 p}-\cos^{p-1}\frac{\pi}{2p}\sin\frac{\pi}{2p}\vartheta_{2p}(\frac{t+\pi}{2})
\\&= 2^{p}  \cos^p\frac{\pi }{2 p}-\cos^{p-1}\frac{\pi}{2p}\sin\frac{\pi}{2p}\vartheta_{1}(t)
.\end{split}\]
This finishes the proof of the case \eqref{mer}.
If \begin{equation}\label{mer1} -\cos t\left(1- \left(\sqrt{2}\sin  \left[\frac{\pi }{2 p}\right]\right)^{\frac{2p}{p-2}}\right)^{-1}> 1,\end{equation}

then $P(r)$ has a stationary point in $(0,1)$. Let $a=\frac{1+r^2}{2r}$. Then $1\le a< \infty$. Define
$$P(r)=Q(a)=\left(-2^{p} a^{p/2} \cos^p\frac{\pi }{2 p}+2^{p/2}\left(a+ \cos t\right)^{p/2}\right).$$ Then $P'(r)=Q'(a)a'(r)$, and so $P'(r)=0$ if and only if $Q'(a)=0$. The stationary point is
$$a_p=\frac{-\cos t}{1- \left(\sqrt{2}\sin  \left[\frac{\pi }{2 p}\right]\right)^{\frac{2p}{p-2}}}.$$

As $a_p\ge 1$ and $$\left( 1-\left(\sqrt{2}\sin  \left[\frac{\pi }{2 p}\right]\right)^{\frac{2p}{p-2}}\right)>0,$$ it follows that $\cos t< 0$.
By assuming without loos of generality  that $0\le t\le 2\pi $, we have from \eqref{mer1} that

$$\frac{\pi}{2}+\sin^{-1}\left[ 1-\left(\sqrt{2}\sin  \left[\frac{\pi }{2 p}\right]\right)^{\frac{2p}{p-2}}\right]<t<\frac{3 \pi}{2} -\sin^{-1}\left[ 1-\left(\sqrt{2}\sin  \left[\frac{\pi }{2 p}\right]\right)^{\frac{2p}{p-2}}\right].$$
We have to prove that

 \begin{equation}\label{QQ}L(t):=Q(a_p)\le -\cos^{p-1}\frac{\pi}{2p}\sin\frac{\pi}{2p} \vartheta_1(t).\end{equation}
Since $L(t)=L(2\pi-t)$ and $\vartheta_1(t)=\vartheta_1(2\pi-t)$, we need to consider only the case $t\in[0,\pi]$, i.e. the case
$$\frac{\pi}{2}+\sin^{-1}\left[ 1-\left(\sqrt{2}\sin  \left[\frac{\pi }{2 p}\right]\right)^{\frac{2p}{p-2}}\right]\leq t<\pi.$$
Let $p_0\approx 2.45$ be the only solution of the equation $$\frac{\pi}{2}+\sin^{-1}\left[ 1-\left(\sqrt{2}\sin  \left[\frac{\pi }{2 p}\right]\right)^{\frac{2p}{p-2}}\right]=\frac{2 \pi }{p}$$ on $[2,\infty)$. Then $$t_p:=\frac{\pi}{2}+\sin^{-1}\left[ 1-\left(\sqrt{2}\sin  \left[\frac{\pi }{2 p}\right]\right)^{\frac{2p}{p-2}}\right]\left\{
                                                                  \begin{array}{ll}
                                                                    \le  \frac{2 \pi }{p}, & \hbox{if $p\le p_0$;} \\
                                                                    >\frac{2 \pi }{p}, & \hbox{if $p>p_0$.}
                                                                  \end{array}
                                                                \right.$$
We divide the rest of the proof into two cases.

\textbf{The case $2\le p\le p_0  \wedge t\le \frac{2\pi}{p}$}.

Since $t_p\le 2\pi/p$,  we have for $t_p\le t\le2\pi/p   $, $$\vartheta_1(t)= -\cos \left[\frac{pt}{2}\right]$$ by taking the substitution $s=t-\pi/2$,  the inequality \eqref{QQ} reduces to

\begin{equation}\label{redu}\begin{split}&\frac{\left(1-2^{p+\frac{p^2}{4-2 p}} \cos \left[\frac{\pi }{2 p}\right] \sin^{\frac{p^2}{2-p}}  \left[\frac{\pi }{2 p}\right]\right)}{\left({-1+2^{\frac{p}{2-p}} \sin^{\frac{2 p}{2-p}}  \left[\frac{\pi }{2 p}\right]}\right)^{p/2}}  \sin^{p/2}  s
\\&\leq \cos \left[\frac{\pi }{2 p}\right]^{p-1}  \sin  \left[\frac{\pi }{2 p}\right]\cos [\frac{p}{2}(\pi/2+s)],\end{split}\end{equation} provided that
\begin{equation}\label{case3}\sin^{-1}\left[ 1-\left(\sqrt{2}\sin  \left[\frac{\pi }{2 p}\right]\right)^{\frac{2p}{p-2}}\right]\leq s<\frac{\pi}{2}. \end{equation}

Let $$v(s)=\cos \left[\frac{1}{2} p \left(\frac{\pi }{2}+s\right)\right]\sin^{-p/2}  s.$$ Then $$v'(s)= \frac{1}{2} p (\sin s)^{-1-\frac{p}{2}} \sin \left[\frac{1}{4} (p-2) (\pi +2 s)\right]\ge 0.$$ So $v$ is increasing.
By plugging $$s=\sin^{-1} \left[1- \left[\sqrt{2} \sin \frac{\pi }{2 p}\right]^{\frac{2 p}{p-2}}\right]$$ in \eqref{redu}, it reduces to the inequality

\[\begin{split}M&=2^p \cos\left[\frac{\pi }{2 p}\right]-2^{\frac{p^2}{2 (p-2)}} \sin^{\frac{p^2}{p-2}} \left[\frac{\pi }{2 p}\right]
\\&+\cos\left[\frac{\pi }{2 p}\right]^{p-1} \cos\left[\frac{p}{2}  \cos^{-1}\left[2^{\frac{p}{p-2}} \sin^{\frac{2 p}{p-2}} \left[\frac{\pi }{2 p}\right]-1\right]\right] \sin \left[\frac{\pi }{2 p}\right]\ge 0.\end{split}\]

For $p\ge 2$, we have $$2^p \cos\left[\frac{\pi }{2 p}\right]\ge 2^2 \frac{\sqrt{2}}{2}=2\sqrt{2}.$$  Further

$$2^{\frac{p^2}{2 (p-2)}} \sin^{\frac{p^2}{p-2}} \left[\frac{\pi }{2 p}\right]=\left(\sqrt{2}\sin\frac{\pi}{2p}\right)^{\frac{p^2}{p-2}}<1$$
and
$$\cos\left[\frac{\pi }{2 p}\right]^{p-1} \cos\left[\frac{p}{2}  \cos^{-1}\left[2^{\frac{p}{p-2}} \sin^{\frac{2 p}{p-2}} \left[\frac{\pi }{2 p}\right]-1\right]\right] \sin \left[\frac{\pi }{2 p}\right]\ge -1.$$

So $$M\ge 2(\sqrt{2}-1)>0.$$

\textbf{The case $(2\le p\le p_0  \wedge t> \frac{2\pi}{p})\vee (p>p_0)$}.
We should prove our inequality \eqref{sharp1} for $$\max\left\{\frac{2\pi}{p}, t_p\right\}\le t\le \pi.$$
By taking the substitution $t=s+\pi/2$, the condition \eqref{mer1} reduces to
\begin{equation}\label{bravo}\max\left\{\frac{2\pi}{p}-\frac{\pi}{2}, t_p-\frac{\pi}{2}\right\} \le s\le \pi/2.\end{equation}
Then
$$\vartheta_1(s+\pi/2)=\max\{|\cos {\frac{p }{2}(s+\pi/2)}|,|\cos {\frac{p}{2}}(3\pi/2-s)|\}$$
As in the previous case we consider the function  $$\Lambda(s)=\frac{\left(1-2^{p+\frac{p^2}{4-2 p}} \cos \left[\frac{\pi }{2 p}\right] \sin  \left[\frac{\pi }{2 p}\right]^{\frac{p^2}{2-p}}\right)}{\left({-1+2^{\frac{p}{2-p}} \sin^{\frac{2 p}{2-p}}  \left[\frac{\pi }{2 p}\right]}\right)^{p/2}}  \sin^{p/2}  s$$ or what is the same
\begin{equation}\label{lama} \Lambda(s)=\frac{\left(1-2^p \cos \left[\frac{\pi }{2 p}\right]\left({\left[\sqrt{2} \sin  \frac{\pi }{2 p}\right]^{\frac{2 p}{2-p}}}\right)^{p/2}\right)}{\left({-1+\left[ \sqrt{2}\sin  \frac{\pi }{2 p}\right]^{\frac{2 p}{2-p}}}\right)^{p/2}}  \sin^{p/2}  s\end{equation} and prove that $$\Lambda(s) \leq -\cos \left[\frac{\pi }{2 p}\right]^{p-1}  \sin  \left[\frac{\pi }{2 p}\right]\vartheta_1(s+\pi/2),$$
provided that \eqref{bravo}. Since $(-1+u)^{-q} \left(1-a u^q\right)$, increases in $u$ for $q>1$, $a>0$ and $u>1$, from \eqref{lama}, it follows that  \begin{equation}\label{lam}\Lambda(s)\le -2^p \cos \left[\frac{\pi }{2 p}\right]\sin^{p/2}s.\end{equation}

If $t_p>\frac{2\pi}{p}$, i.e. if $p>p_0$, then for $t_p-\frac{\pi}{2}\le s\le \pi/2$ we have
\[\begin{split}\Lambda(s)&\le -2^p \cos \left[\frac{\pi }{2 p}\right]\sin^{p/2}(t_p-\frac{\pi}{2})\\&\le -4 \cos\left[\frac{\pi }{2 p}\right] \left(1-2^{\frac{p}{p-2}} \sin^{\frac{2 p}{p-2}} \left[\frac{\pi }{2 p}\right]\right)^{p/2}.\end{split}\]
So $$\Lambda(s)< -2\le -1\le -\cos \left[\frac{\pi }{2 p}\right]^{p-1}  \sin  \left[\frac{\pi }{2 p}\right]\vartheta_1(s+\pi/2).$$

The only case remained is $t_p\le \frac{2\pi}{p}\wedge t>\frac{2\pi}{p}$. So we consider the remained case $p\le p_0$ and $$\frac{2\pi}{p}-\frac{\pi}{2}\le s\le \frac{\pi}{2}.$$

We also have $$\Lambda(s)\le   -2^p \cos \left[\frac{\pi }{2 p}\right]\sin^{p/2}(\frac{2\pi}{p}-\frac{\pi}{2})\le -2^p \cos \left[\frac{\pi }{2 p}\right]\sin^{p/2}(t_p-\frac{\pi}{2})\le -2.$$ The rest is the same as the previous part of the proof.

 This finishes the proof.
\end{proof}

\begin{proof}[Proof of Lemma~\ref{vanesa}]
Let $$a=\frac{1+r^2}{2r}.$$ Then we should prove that $$\lambda(a,t)=(a+\cos t)^{\frac{p}{2}}+2^{\frac{p}{2}} \cos \frac{\pi }{2 p} \cos \left[\frac{p}{2}  (\pi -t)\right] \sin^{p-1} \frac{\pi }{2 p}-(2a)^{\frac{p}{2}}  \sin^p\frac{\pi }{2 p}\le 0$$ for $(a,t)\in K:=[1,\infty)\times [0,\pi]$.

 We first show that $\lambda$ has not stationary points in the interior of $K$.
 We have $$\lambda_t = \frac{1}{2} p \left(2^{p/2} \cos \frac{\pi }{2 p} \sin^{p-1} \frac{\pi }{2 p} \sin \left[\frac{1}{2} p (\pi -t)\right]-(a+\cos t)^{-1+\frac{p}{2}} \sin t\right)$$ and

$$\lambda_a = \frac{1}{2} p (a+\cos t)^{-1+\frac{p}{2}}-2^{-1+\frac{p}{2}} a^{-1+\frac{p}{2}} p \sin^p\frac{\pi }{2 p}.$$

If $\lambda_a=0$, then $a= \frac{\cos t}{-1+\left(2^{p/2} \sin^p\frac{\pi }{2 p}\right)^{\frac{2}{p-2}}}$.
As
$\sqrt{2}\sin\frac{\pi}{2p}\ge 1$, it follows that $${-1+\left(2^{p/2} \sin^p\frac{\pi }{2 p}\right)^{\frac{2}{p-2}}}<0,$$ and thus $a\ge 0$ if $t\in[\pi/2,\pi]$. So $\lambda_t=0$ if and only if \[\begin{split} Z=2^{p/2} &\cos \frac{\pi }{2 p} \sin^{p-1} \frac{\pi }{2 p} \sin \left[\frac{1}{2} p (\pi -t)\right]\\&-(-\cos t)^{\frac{1}{2} (p-2)} \left(-1+\frac{1}{1-\left(2^{p/2} \sin^p\frac{\pi }{2 p}\right)^{\frac{2}{p-2}}}\right)^{\frac{1}{2} (p-2)} \sin t=0.\end{split}\]
Further let $$X:=2^{p/2}\cos \left[\frac{\pi }{2p}\right]\sin^{p-1} \left[\frac{\pi }{2p}\right]$$ and $$Y :=\left(-1+\frac{1}{1-\left(2^{p/2} \sin^p\frac{\pi }{2 p}\right)^{\frac{2}{p-2}}}\right)^{\frac{1}{2} (p-2)}.$$
In order to continue prove the following lemma
\begin{lemma}\label{lpe}
If $1\le p\le 2$, we have $X\le 1$ and $Y\ge 1$.
\end{lemma}
\begin{proof}[Proof of Lemma~\ref{lpe}]
First of all $$X= 2^{p/2} \cos^{p/2}\left[\frac{\pi }{2 p}\right]\sin ^{p/2}\left[\frac{\pi }{2 p}\right]\cdot\cot\left[\frac{\pi }{2 p}\right]^{1-\frac{p}{2}}.$$ So

$$X= \sin ^{p/2}\left[\frac{\pi }{p}\right] \cot^{1-\frac{p}{2}}\left[\frac{\pi }{2 p}\right]\le 1.$$

Prove now that  $Y\ge 1$. It is equivalent with the inequality $$2^{\frac{p}{p-2}} \sin ^{\frac{2 p}{p-2}}\left[\frac{\pi }{2 p}\right]\le1/2,\ \ \ 1\le p\le 2,$$ or what is the same
\begin{equation}\label{before}2^{1-\frac{1}{p}} \sin \left[\frac{\pi }{2 p}\right]\geq 1, \ \ \ 1\le p\le 2.\end{equation} The last inequality is equivalent with $$2 \sin \left[\frac{\pi }{2 p}\right]>2^{\frac{1}{p}},\ \ \ \ \ \ 1\le p <2.$$ By taking the substitution $x=1/p$, it reduces to the inequality
$$g(x):=\log \left[2 \sin\left[\frac{\pi  x}{2}\right]\right]-x \log 2\ge 0,\ \ \ 1/2\le x\le 1.$$ Further $$g'(x)=\frac{1}{2} \pi  \cot\left[\frac{\pi  x}{2}\right]-\log 2$$ and $$g''(x)=-\frac{1}{4} \pi ^2 \csc^2 \left[\frac{\pi  x}{2}\right].$$ So $$g(x)\ge \min\{g(1),g(1/2)\}= 0.$$
This implies that $X\ge 1$.

\end{proof}

It follows that
\[\begin{split}Z&=X\sin \left[\frac{1}{2} p (\pi -t)\right]-(-\cos t)^{\frac{1}{2} (p-2)} Y\sin t\\&
\le \sin \left[\frac{1}{2} p (\pi -t)\right]-(-\cos t)^{\frac{1}{2} (p-2)} \sin t \\&<\sin \left[\frac{1}{2} p (\pi -t)\right]- \sin t<0 .\end{split}\]
Further we estimate  $\lambda$ on $\partial K$. We have  $$\lambda(a,0)=(1+a)^{p/2}+2^{p/2} \cos \frac{\pi }{2 p} \cos \left[\frac{p \pi }{2}\right] \sin \frac{\pi }{2 p}^{p-1}-2^{p/2} a^{p/2} \sin^p\frac{\pi }{2 p}$$ and so $$\partial_a\lambda(a,0)=\frac{1}{2} (1+a)^{-1+\frac{p}{2}} p-2^{-1+\frac{p}{2}} a^{-1+\frac{p}{2}} p \sin^p\frac{\pi }{2 p}\le 0.$$
Furthermore $$\lambda(a,\pi)=(-1+a)^{p/2}+2^{p/2} \cos \frac{\pi }{2 p} \sin \frac{\pi }{2 p}^{p-1}-2^{p/2} a^{p/2} \sin^p\frac{\pi }{2 p}.$$ Thus $$\partial_a\lambda(a,\pi)=\frac{1}{2} (-1+a)^{-1+\frac{p}{2}} p-2^{-1+\frac{p}{2}} a^{-1+\frac{p}{2}} p \sin^p\frac{\pi }{2 p}.$$ So $\partial_a\lambda(a,\pi)=0$ if and only if $$a_p=-\frac{1}{-1+\left(2^{p/2} \sin^p\frac{\pi }{2 p}\right)^{\frac{2}{p-2}}}.$$ In this case $$\lambda(a_p,\pi)=2^{p/2} \cos \frac{\pi }{2 p} \sin^{p-1} \frac{\pi }{2 p}-\left({-1+2^{\frac{p}{2-p}} \sin^{\frac{2 p}{2-p}} \frac{\pi }{2 p}}\right)^{\frac{1}{2} (2-p)}\le 0,$$ because $$2^{p/2} \cos \frac{\pi }{2 p} \sin^{p-1} \frac{\pi }{2 p}\le \sin\frac{\pi}{p}<1,\ \ \ \ \ 1\le p <2$$ and
$$2^{\frac{p}{2-p}} \sin^{\frac{2 p}{2-p}} \frac{\pi }{2 p}>2, \ \ \ \ \ \ 1\le p <2$$ but this is proved before \eqref{before}.

It remains to consider the cases $\lim_{a\to +\infty}\lambda(a,t)$  and $\lambda(1,t)$. First of all $$\lim_{a\to +\infty}\lambda(a,t)=-\infty.$$ Further we have $$\lambda(1,t)=(1+\cos t)^{p/2}+2^{p/2} \cos \frac{\pi }{2 p} \cos \left[\frac{1}{2} p (\pi -t)\right] \sin^{p-1} \frac{\pi }{2 p}-2^{p/2} \sin^p\frac{\pi }{2 p}.$$ Then $\lambda(1,t)\le 0$ if and only if $$\mu(t):=\cos \left[\frac{t}{2}\right]^p+\cos \frac{\pi }{2 p} \cos \left[\frac{1}{2} p (\pi -t)\right] \sin^{p-1} \frac{\pi }{2 p}-\sin^p\frac{\pi }{2 p}\le 0.$$ Further $$\mu(t)\le 0$$ if and only if $$\eta(t):=\frac{\cos \left[\frac{t}{2}\right]^p \sin^{-p} \frac{\pi }{2 p}}{1-\cos \left[\frac{1}{2} p (\pi -t)\right] \cot \frac{\pi }{2 p}}\le 1.$$

Next we have $$\eta'(t)=\frac{p \cos \left[\frac{t}{2}\right]^{p-1} \sin^{-1-p} \frac{\pi }{2 p} \zeta(t)}{2 \left(-1+\cos \left[\frac{1}{2} p (\pi -t)\right] \cot \frac{\pi }{2 p}\right)^2},$$ where $$\zeta(t)=\left(\cos \frac{\pi }{2 p} \cos \left[\frac{1}{2} (p-1) (\pi -t)\right]-\sin \frac{\pi }{2 p} \sin \left[\frac{t}{2}\right]\right).$$
Since $$\zeta'(t)=-\frac{1}{2} \cos \left[\frac{t}{2}\right] \sin \frac{\pi }{2 p}-\frac{1}{2} (1-p) \cos \frac{\pi }{2 p} \sin \left[\frac{1}{2} (p-1) (\pi -t)\right],$$ we obtain that \[\begin{split}\zeta'(t)&\le  \frac{1}{2}\sin\frac{\pi}{2p}\left(- \cos \left[\frac{t}{2}\right] -\frac{1}{2} (1-p)\text{  }\sin \left[\frac{1}{2} (p-1) (\pi -t)\right]\right)
\\&\le \frac{1}{2}\sin\frac{\pi}{2p}\left(- \cos \left[\frac{t}{2}\right] -\text{  }\sin \left[\frac{1}{2} (p-1) (\pi -t)\right]\right)
\\&\le  \frac{1}{2}\sin\frac{\pi}{2p}\left(- \cos \left[\frac{t}{2}\right] -\text{  }\sin \left[\frac{1}{2} (-1+2) (\pi -t)\right]\right)=0.\end{split}\]
Thus $\zeta$ is strictly decreasing. Hence $\eta'(t)$ is strictly decreasing. Since $\eta'((1-1/p)\pi)=0$, it follows that $(1-1/p)\pi$ is the only zero of $\eta'$, and thus $\eta$ attains its maximum for $t=(1-1/p)\pi$, which is equal to $1$. This finishes the proof of lemma.
\end{proof}

\section{Appendix}
For completeness we prove the following lemma of Verbitsky (\cite{ver}), needed for the proof of Lemma~\ref{posa}.
\begin{lemma}\label{versi} Let $1<p\le 2$.
For $A=A(p)=\frac{1}{\cos^p\frac{\pi}{2p}}$ and $B=B(p)=\tan\frac{\pi}{2p}$ and $|x|\le \frac{\pi}{2}$ we have
$$A\cos^p x - B \cos(px)\ge 1.$$

\end{lemma}
\begin{proof}
Let $x\ge 0$ and define   $$f(x) =\frac{ (1+B \cos(px))}{\cos^p x}. $$ Then $$f'(x)= p \frac{\sin x+B \sin (x-p x)}{\cos^{1+p} x }.$$

We need to show that the only solution of $$g(x)=\sin x+B \sin (x-p x)=0$$ is $x=\pi/(2p)$. Show that $g$ is concave.
We have $$g''(x)=-\sin x+(1-p)^2 \sin((p-1)x) \tan\frac{\pi }{2 p}.$$

As $$(1-p)^2\tan\frac{\pi }{2 p}\le (p-1)\tan\frac{\pi }{2 p} \le 1,$$ we have that $g''(x)\le -\sin x+ \sin((p-1)x) \le 0$. As $g'(0)=1+(1-p) \tan \frac{\pi }{2 p}>0$ and
$$g'(\pi/2)=(1-p) \sin\left[\frac{p \pi }{2}\right] \tan \frac{\pi }{2 p}<0,$$ there is a unique solution $x_0$ of $g'(x)=0$. Then $g$ increases on $(0,x_0)$ and decreases on $(x_0,\pi/2)$. Further $g(x_0)>0=g(\pi/(2p))> g(\pi/2)=1+\cos \left[\frac{p \pi }{2}\right] \tan \frac{\pi }{2 p}$. Since $f(0)=0$ and $\lim_{x\uparrow \pi/2} f(x)=-\infty$, we obtain that  $f(x)\le f(\pi/(2p))=A$. This finishes the proof.
\end{proof}
\subsection*{Acknowledgement} I am grateful to professor  I. E. Verbitsky for emailing references that have had a significant impact on this paper.

\end{document}